\documentclass[12pt]{amsart}
\usepackage{preamble}

\usepackage[headheight=15pt, headsep=15pt, footskip=27pt, bottom=2.4cm, left=2.4cm, right=2.4cm]{geometry}

\begin{document}

\title[On the relative Morrison-Kawamata cone conjecture]{On the relative Morrison-Kawamata cone conjecture}


\begin{abstract}
We relate the Morrison-Kawamata cone conjecture for Calabi-Yau fiber spaces to the existence of Shokurov polytopes. For K3 fibrations, the existence of (weak) fundamental domains for movable cones is established. The relationship between the relative cone conjecture and the cone conjecture for its geometric or generic fibers is studied.
\end{abstract}

\author{Zhan Li}
\address[Zhan Li]{Department of Mathematics, Southern University of Science and Technology, 1088 Xueyuan Rd, Shenzhen 518055, China} \email{lizhan@sustech.edu.cn, lizhan.math@gmail.com}

\author{Hang Zhao}
\address[Hang Zhao]{School of Mathematics and Statistics, Yunnan University, Kunming 650091, China} \email{zhaoh@ynu.edu.cn}

\maketitle

\tableofcontents

\section{Introduction}

The purpose of this paper is to study the following (relative) Morrison-Kawamata cone conjecture \cite{Mor93,Mor96,Kaw97,Tot09}.

\begin{conjecture}\label{conj: KM conj}
Let $(X, \De) \to S$ be a klt Calabi-Yau fiber space. Let $\Gamma_B, \Gamma_A$ be the images of the pseudo-automorphism group $\PsAut(X/S,\De)$ and the automorphism group $\Aut(X/S,\De)$ under the natural group homomorphism $\PsAut(X/S,\De) \to {\rm GL}(N^1(X/S)_\Rr)$ respectively.
\begin{enumerate}
\item The cone $\bMov^e(X/S)$ has a (weak) rational polyhedral fundamental domain under the action of $\Gamma_B$.
\item The cone $\bAmp^e(X/S)$ has a  (weak) rational polyhedral fundamental domain under the action of $\Gamma_A$.
\end{enumerate}
\end{conjecture}

Relevant notions in Conjecture \ref{conj: KM conj} are explained in Section \ref{sec: preliminaries} and Section \ref{sec: Geometry of convex cones}. In particular, the (weak) rational polyhedral fundamental domain is defined in Definition \ref{def: fundamental domain}. There are different choices of  cones in the cone conjecture, see Remark \ref{rmk: choice of cones} for the reason of the above choice.

At the expense of some ambiguity, for simplicity, we call Conjecture \ref{conj: KM conj} (1) and (2) the (weak) cone conjecture for movable cones and  the (weak) cone conjecture for ample cones respectively. Although our primary interest is in complex varieties, we need to work with non-algebraically closed fields. When $X$ is a smooth Calabi-Yau variety over a field $K$, the analogous cone conjecture still makes sense, and we also call it the cone conjecture. 

The cone conjecture is beyond merely predicting the shape of cones of Calabi-Yau varieties. In fact, for an arbitrary klt pair, if $(X,\De)$ is its minimal model and $(X,\De) \to S$ is the morphism to its canonical model, then the cone conjecture for movable cones predicts finiteness of minimal models (see Proposition \ref{prop: fun domain of Mov} for the precise statement). Moreover, compared with the weaker prediction of having only finitely many $\PsAut(X/S,\De)$- or $\Aut(X/S,\De)$-equivalence classes, the existence of (weak) fundamental domains provides additional information that is crucial for the proof of the cone conjecture (for example, in the proof of Proposition \ref{prop: realize actions}, we rely on the finite generation of $\bar \Gamma_B$).

When $X \to S$ is a birational morphism, \cite{BCHM10} established the finiteness of $\PsAut(X/S)$-equivalence classes. Finiteness of $\PsAut(X/S)$-equivalence classes is also known when $\dim X \leq 3, \dim S >0$ (\cite{Kaw97}) and elliptic fibrations (\cite{FHS21}). When $S$ is a point, Conjecture \ref{conj: KM conj} is known for surfaces (\cite{Tot09}), abelian varieties (\cite{PS12}) and large classes of Calabi-Yau manifolds with Picard number $2$ (\cite{Ogu14, LP13}). The analogous cone conjecture for $\Mov(X/\Cc)_+$ (see Definition \ref{def: polyhedral type}) is also known in the case of a projective hyperk\"ahler manifold $X$ \cite{Mar11}. See also \cite{HPX24} for the cone conjecture for families of irreducible holomorphic symplectic manifolds. Over arbitrary fields of characteristic $\neq 2$, the cone conjecture is known for K3 surfaces \cite{BLL20}.  Analogous cone conjecture of $\Mov(X/K)_+$ is also known for  a hyperk\"ahler variety $X$ over a field $K$ with characteristic $0$ (\cite{Tak21}, cf. Remark \ref{rmk: hyperkahler}). On the other hand, it is known that Conjecture \ref{conj: KM conj} no longer holds true for lc pairs (see \cite{Tot09}). We recommend \cite{LOP18} for a survey of relevant results. Using some of the ideas developed in the present paper, \cite[Theorem~14]{Xu24} proves that the cone conjecture for ample cones follows from the cone conjecture for movable cones. This result was further extended in \cite{GLSW24} to the case of the effective cone.

The new ingredient of the present paper is to study the cone conjecture from the perspective of Shokurov polytopes. We propose the following conjecture which seems to be more tractable.

\begin{conjecture}\label{conj: shokurov polytope}
Let $f: (X, \De) \to S$ be a klt Calabi-Yau fiber space. 
\begin{enumerate}
\item There exists a polyhedral cone $P_M \subset \Eff(X/S)$ such that $$\qquad \bigcup_{g \in \PsAut(X/S, \De)} g\cdot P_M \supset  \Mov(X/S).$$
\item There exists a polyhedral cone  $P_A \subset \Eff(X/S)$ such that $$ \qquad \bigcup_{g \in \Aut(X/S, \De)} g\cdot P_A \supset  \Amp(X/S).$$
\end{enumerate}
\end{conjecture}

It seems that Conjecture \ref{conj: shokurov polytope} is more fundamental, as it incorporates both the finiteness of models or contractions and the existence of fundamental domains. This perspective is further reinforced by the work of \cite{Xu24, GLSW24}.

Using results of \cite{Loo14} and assuming standard conjectures of log minimal model program (LMMP), we show that Conjecture \ref{conj: shokurov polytope} is nearly equivalent to the cone conjecture (when $S$ is a point, they are indeed equivalent).

\begin{theorem}\label{thm: main 1}
Let $f: (X,\De) \to S$ be a klt Calabi-Yau fiber space. 
\begin{enumerate}
\item Assume that good minimal models exist for effective klt pairs in dimension $\dim(X/S)$. If $R^1f_*\Oo_X=0$, then the weak cone conjecture for $\bMov^e(X/S)$ is equivalent to the Conjecture \ref{conj: shokurov polytope} (1).
\item Assume that good minimal models exist for effective klt pairs in dimension $\dim(X/S)$. If $\bMov(X/S)$ is non-degenerate, then the cone conjecture for $\bMov^e(X/S)$ is equivalent to the Conjecture \ref{conj: shokurov polytope} (1).
\item The cone conjecture for $\bAmp^e(X/S)$ is equivalent to the Conjecture \ref{conj: shokurov polytope} (2).
\end{enumerate}
\end{theorem}

Using this circle of ideas, we study a Calabi-Yau fiber space $X \to S$ fibered by K3 surfaces.  This means that for a general closed point $t\in S$, its fiber $X_t$ is a smooth K3 surface. We establish the (weak) cone conjecture of $\bMov^e(X/S)$ for K3 fibrations.

\begin{theorem}\label{thm: K3}
Let $f: X \to S$ be a Calabi-Yau fiber space such that $X$ has terminal singularities. 

If $f$ is fibered by K3 surfaces, then the weak cone conjecture of $\bMov^e(X/S)$ holds true. 

Moreover, if $\bMov(X/S)$ is non-degenerate, then the cone conjecture holds true for $\bMov^e(X/S)$. In particular, if $S$ is $\Qq$-factorial, then the cone conjecture holds true for $\bMov^e(X/S)$. 
\end{theorem}

In the subsequent paper \cite{Li23}, we establish the weak cone conjecture for movable cones of terminal Calabi-Yau fibrations in relative dimension $\leq 2$. This is partially extended to klt Calabi-Yau fibrations in relative dimension two by \cite{MS24}.

We discuss the contents of the paper. Section \ref{sec: preliminaries} gives the necessary background materials and fixes notation. Section \ref{sec: Geometry of convex cones} develops the geometry of convex cones following \cite{Loo14}. Section \ref{sec: Generic properties of fibrations and structures of cones} establishes properties of generic and geometric fibers which will be used in Section \ref{sec: Generic and Geometric cone conjecture}. Section \ref{sec: A variant of the cone conjecture} studies the relationship between the cone conjecture and Conjecture \ref{conj: shokurov polytope}. In particular, Theorem \ref{thm: main 1} is proven. Section \ref{sec: Generic and Geometric cone conjecture} studies the cone conjecture by assuming that it holds true for geometric or generic fibers. Theorem \ref{thm: K3} is shown in Section \ref{subsec: generic cone conj}.

\medskip

\noindent {\bf Acknowledgements.} We benefit from discussions with Lie Fu, Yong Hu, Vladimir Lazi\'{c}, Zhiyuan Li, Chen Jiang, Yannan Qiu, Hao Sun, and Jinsong Xu. We thank Xingying Li for pointing out a mistake in Lemma \ref{le: lift to iso in codim 1} and indicating the method to fix it. We are grateful to the anonymous referee for valuable and constructive suggestions. Zhan Li is partially supported by the NSFC No.12471041 and the Guangdong Basic and Applied Basic Research Foundation No.2024A1515012341. Hang Zhao is partially supported by the Scientific Research and Innovation Fund of Yunnan University No.ST20210105. Both authors are partially supported by a grant from SUSTech.

\section{Preliminaries}\label{sec: preliminaries}

Let $f: X \to S$ be a projective morphism between normal quasi-projective varieties over $\Cc$. Then $f$ is called a fibration if $f$ is surjective with connected fibers. We write $X/S$ to mean that $X$ is over $S$. 

By divisors, we mean Weil divisors. For $\mathbb K=\Zz, \Qq, \Rr$ and two $\mathbb K$-divisors $A, B$ on $X$, $A \sim_{\mathbb K} B/S$ means that $A$ and $B$ are $\mathbb K$-linearly equivalent over $S$. If $A, B$ are $\Rr$-Cartier divisors, then $A \equiv B/S$ means that $A$ and $B$ are numerically equivalent over $S$. 

We use $\Supp E$ to denote the support of the divisor $E$. A divisor $E$ on $X$ is called a vertical divisor (over $S$) if $f(\Supp E) \neq S$. A vertical divisor $E$ is called a very exceptional divisor if for any prime divisor $P$ on $S$, over the generic point of $P$, we have $\Supp f^*P \not\subset \Supp E$ (see \cite[Definition 3.1]{Bir12b}). If $f$ is a birational morphism, then the notion of very exceptional divisor coincides with that of exceptional divisor. 

Let $X$ be a normal complex variety and $\De$ be an $\Rr$-divisor on $X$, then $(X, \De)$ is called a log pair. We assume that $K_X+\De$ is $\Rr$-Cartier for a log pair $(X, \De)$. Then $f: (X, \De) \to S$ is called a Calabi-Yau fibration/fiber space if $X \to S$ is a fibration, $X$ is $\Qq$-factorial and $K_X+\De \sim_\Rr 0/S$. When $(X, \De)$ has lc singularities (see Section \ref{subsec: minimal models}),  then $K_X+\De \sim_\Rr 0/S$ is equivalent to the weaker condition  $K_X+\De \equiv 0/S$ by \cite[Corollary 1.4]{HX16}.

\subsection{Movable cones and ample cones}\label{subsec: Movable cones and ample cones}

Let $V$ be a finite-dimensional real vector space with a rational structure, that is, a $\Qq$-vector subspace $V(\Qq)$ of $V$ such that $V=V(\Qq)\otimes_\Qq\Rr$. A set $C\subset V$ is called a cone if for any $x\in C$ and $\lambda\in \Rr_{>0}$, we have $\lambda \cdot x \in C$. We use $\Int(C)$ to denote the relative interior of $C$ and call  $\Int(C)$ the relatively open cone. By convention, the origin is a relatively open cone.  A cone is called a polyhedral cone (resp. rational polyhedral cone) if it is a closed convex cone generated by finite vectors (resp. rational vectors). If $S \subset V$ is a subset, then $\Conv(S)$ denotes the convex hull of $S$, and $\Cone(S)$ denotes the closed convex cone generated by $S$. As we are only concerned about convex cones in this paper, we also call them cones.

Let $\Pic(X/S)$ be the relative Picard group. Let $$N^1(X/S) \coloneqq \Pic(X/S)/{\equiv}$$ be the lattice. Set $\Pic(X/S)_{\mathbb K} \coloneqq \Pic(X/S) \otimes_\Zz \mathbb K$ and $N^1(X/S)_{\mathbb K} \coloneqq N^1(X/S) \otimes_\Zz \mathbb K$ for $\mathbb K = \Qq$ or $\Rr$. If $D$ is an $\Rr$-Cartier divisor, then $[D] \in N^1(X/S)_\Rr$ denotes the corresponding divisor class. To abuse the terminology, we also call $[D]$ an $\Rr$-Cartier divisor.

Recall that an $\Rr$-Cartier divisor $D$ is effective$/S$ if there exists an effective divisor $E \geq 0$ such that $D \sim_\Rr E/S$. A Cartier divisor $D$ is movable$/S$ if the base locus of the relative linear system $|D/S|$ has codimension $>1$. We list relevant cones inside $N^1(X/S)_\Rr$ which appear in the paper:
\begin{enumerate}
\item $\Eff(X/S)$: the cone generated by effective$/S$ Cartier divisors;
\item $\bEff(X/S)$: the closure of $\Eff(X/S)$;
\item $\Mov(X/S)$: the cone generated by movable$/S$ divisors;
\item $\bMov(X/S)$: the closure of $\Mov(X/S)$;
\item $\bMov^e(X/S) \coloneqq \bMov(X/S) \cap \Eff(X/S)$;
\item ${\Mov(X/S)_+}\coloneqq \Conv(\bMov(X/S) \cap N^1(X/U)_\Qq)$ (see Definition \ref{def: polyhedral type});
\item $\Amp(X/S)$: the cone generated by ample$/S$ divisors;
\item $\bAmp(X/S)$: the closure of $\Amp(X/S)$;
\item $\bAmp^e(X/S) \coloneqq \bAmp(X/S) \cap \Eff(X/S)$;
\item  $\Amp(X/S)_+ \coloneqq\Conv(\bAmp(X/S) \cap N^1(X/U)_\Qq)$.
\end{enumerate}

If $K$ is a field of characteristic zero and $X$ is a variety over $K$, then the above cones still make sense for $X$. We use $ \Mov(X/K), \Amp(X/K)$, etc. to denote the corresponding cones.

Recall that for a birational map $g: X \dto Y/S$, if $D$ is an $\Rr$-Cartier divisor on $X$, then the pushforward of $D$, $g_*D$, is defined as follows.  Let $p: W \to X, q: W \to X$ be birational morphisms such that $g \circ p=q$, then $g_*D \coloneqq q_*(p^*D)$. This is independent of the choice of $p$ and $q$. 

 Let $\De$ be a divisor on a $\Qq$-factorial variety $X$. We use $\Bir(X/S, \De)$ to denote the birational automorphism group of $(X, \De)$ over $S$. To be precise, $\Bir(X/S, \De)$ consists of birational maps $ g: X \dto X/S$ such that $g_* \Supp \De= \Supp \De$. A birational map is called a pseudo-automorphism if it is isomorphic in codimension $1$. Let $\PsAut(X/S, \De)$ be the subgroup of  $\Bir(X/S, \De)$ consisting of pseudo-automorphisms. Let $\Aut(X/S, \De)$ be the subgroup of $\Bir(X/S, \De)$ consisting of automorphisms of $X/S$. For a field $K$, if $X$ is a variety over $K$ and $\De$ is a divisor on $X$, then we still use $\Bir(X/K, \De), \PsAut(X/K, \De)$ and $\Aut(X/K, \De)$ to denote the birational automorphism group, the pseudo-automorphism group and the automorphism group of $X/K$ respectively.
 
Let $g\in \Bir(X/S, \De)$ and $D$ be an $\Rr$-Cartier divisor on a $\Qq$-factorial variety $X$. Because the pushforward map $g_*$ preserves numerical equivalence classes, there is a linear map
\[
g_*: N^1(X/S)_\Rr \to N^1(X/S)_\Rr, \quad [D] \mapsto [g_*D].
\] It is straightforward to check that
\[
\begin{split}
\PsAut(X/S, \De) \times N^1(X/S)_\Rr &\to N^1(X/S)_\Rr\\
(g, [D])& \mapsto [g_*D],
\end{split}
\] is a group action. We use $g \cdot D, g \cdot [D]$ to denote $g_*D, [g_*D]$ respectively. Let $\Gamma_B$ and $\Gamma_A$ be the images of $\PsAut(X/S, \De)$ and $\Aut(X/S, \De)$ under the natural group homomorphism $$\iota: \PsAut(X/S, \De) \to {\rm GL}(N^1(X/S)_\Rr).$$ Because $\Gamma_B, \Gamma_A \subset {\rm GL}(N^1(X/S))$, $\Gamma_B$ and $\Gamma_A$ are discrete subgroups. By abusing the notation, we also write $g$ for $\iota(g) \in \Gamma_B$, and denote $\iota(g)([D])$ by $g\cdot [D]$. Then the cones $\Mov(X/S), \bMov(X/S), \bMov^e(X/S)$ and $\Mov(X/S)_+$ are all invariant under the action of $\PsAut(X/S,\De)$. Similarly, $\Amp(X/S), \Amp(X/S), \Amp^e(X/S)$ and $\Amp(X/S)_+$ are all invariant under the action of $\Aut(X/S,\De)$.

\begin{remark} 
If $g\in \Bir(X/S)$ is not isomorphic in codimension $1$, then for $[D] \in \Mov(X/S)$, $[g_*D]$ may not be in $\Mov(X/S)$. Moreover, $(g , [D]) \mapsto [g_*D]$ is not a group action of $\Bir(X/S, \De)$ on $N^1(X/S)_\Rr$. For one thing, if $D$ is a divisor contracted by $g$, then $g^{-1}_*(g_*[D])= 0 \neq (g^{-1}\circ g)_*[D]$.
\end{remark}

The following example gives a birational map that is not a pseudo-automorphism. 

\begin{example}
Let $f(x,y,z)$ be a general homogeneous cubic polynomial. Let $D \coloneqq \{f(x,y,z)=0\}\subset \Pp^2$ and $B\coloneqq \{f(-x,y,z)=0\}\subset \Pp^2$. Then $(\Pp^2, \frac 1 2 D+ \frac 1 2 B)$ is a klt Calabi-Yau pair. Let $$p_1=[a:b:c], p_2=[-a:b:c] \in D \cap B$$ be two distinct points. Let $\pi_i: X_i \to \Pp^2, i=1,2$ be the blowing up of $p_i$ such that $E_i, i=1,2$ are corresponding exceptional divisors. If $D_i, B_i$ are the strict transforms of $D, B$ on $X_i$, then
\[
K_{X_i}+\frac 1 2 D_i+ \frac 1 2 B_i=\pi_i^*(K_{\Pp^2}+\frac 1 2 D+ \frac 1 2 B).
\] Therefore, each $({X_i}, \frac 1 2 D_i+ \frac 1 2 B_i)$ is a klt Calabi-Yau pair. Moreover, $({X_1}, \frac 1 2 D_1+ \frac 1 2 B_1)$ is isomorphic to $({X_2}, \frac 1 2 D_2+ \frac 1 2 B_2)$ through $\pi_2^{-1}\circ\tau\circ\pi_1$, where $\tau: \Pp^2 \to \Pp^2$ is given by $[x:y:z] \mapsto [-x:y:z]$. However, the birational map
\[
\pi_2^{-1}\circ\pi_1: (X_1, \frac 1 2 D_1+ \frac 1 2 B_1) \dto (X_2, \frac 1 2 D_2+ \frac 1 2 B_2)
\] is not isomorphic in codimension $1$. In fact, this map contracts $E_1$ and extracts $E_2$.
\end{example}

\subsection{Minimal models of varieties}\label{subsec: minimal models}

Let $(X,\De)$ be a log pair. For a divisor $D$ over $X$, if $f: Y \to X$ is a birational morphism from a smooth variety $Y$ such that $D$ is a prime divisor on $Y$, then the log discrepancy of $D$ with respect to $(X, \De)$ is defined to be $$a(D; X, \De) \coloneqq\mult_{D}(K_Y-f^*(K_X+\De))+1.$$ This definition is independent of the choice of $Y$. A log pair $(X,\De)$ (or its singularity) is called sub-klt (resp. sub-lc) if the log discrepancy of any divisor over $X$ is $>0$ (resp. $\geq 0$). If $\De \geq 0$, then a sub-klt (resp. sub-lc) pair $(X,\De)$ is called klt (resp. lc). If $\De=0$ and the log discrepancy of any exceptional divisor over $X$ is $>1$, then $X$ is said to have terminal singularities. A fibration/fiber space $(X, \De) \to S$ is called a klt (resp. terminal) fibration/fiber space if $(X, \De)$ is klt (resp. terminal). In the sequel, we will use a well-known fact that if $X$ has terminal singularities with $K_X$ nef$/S$, then $\Bir(X/S) = \PsAut(X/S)$ (see, for example, \cite[Lemma 2.6]{Li23}).

Let $X \to S$ be a projective morphism of normal quasi-projective varieties. Suppose that $(X, \De)$ is klt. Let $\phi: X \dto Y/S$ be a birational contraction (i.e. $\phi$ does not extract divisors) of normal quasi-projective varieties over $S$, where $Y$ is projective over $S$. We write $\De_Y \coloneqq \phi_*\De$ for the strict transform of $\De$. Then $(Y/S, \De_Y)$ is a weak log canonical model of $(X/S, \De)$ if $K_Y+\De_Y$ is nef$/S$ and $a(D; Y, \De_Y) \geq a(D;X, \De)$ for any divisor $D$ over $X$.

\begin{lemma}\label{le: lift to iso in codim 1}
Let $(X/S, \De)$ be a klt pair with $[K_X+\De] \in \bMov(X/S)$. Suppose that $g: (X/S, \De) \dto (Y/S, \De_Y)$ is a weak log canonical model of $(X/S, \De)$. Then $(X/S, \De)$ admits a weak log canonical model $(Y'/S, \De_{Y'})$ such that 
\begin{enumerate}
\item $Y'$ is $\Qq$-factorial,
\item $X, Y'$ are isomorphic in codimension $1$, and
\item there exists a morphism $\nu: Y' \to Y/S$ such that $K_{Y'}+\De_{Y'} = \nu^*(K_Y+\De_Y)$.
\end{enumerate}
\end{lemma}
\begin{proof}
If $E$ is a prime divisor on $X$ which is exceptional over $Y$ and 
\[
a(E; X, \De) = a(E; Y, \De_Y),
\] then by $a(E; X, \De) \leq 1$, we have  $a(E; Y, \De_Y) \leq 1$. By \cite[Corollary 1.4.3]{BCHM10}, there exist a $\Qq$-factorial variety $Y'$ and a birational morphism $\nu: Y' \to Y$ which extracts all such divisors. We have $K_{Y'}+\De_{Y'} = \nu^*(K_Y+\De_Y)$ with $\De_{Y'} \geq 0$. Moreover, if $E$ is an exceptional divisor for $\nu^{-1}\circ g$, then 
\begin{equation}\label{eq: <}
a(E; X, \De)< a(E; Y', \De_{Y'}).
\end{equation}

It suffices to show that $X \dto Y'$ is isomorphic in codimension $1$. Let $p: W \to X$ and $q: W \to Y'$ be birational morphisms such that $q \circ p^{-1}=\nu^{-1}\circ g$. Then we have
\[
p^*(K_X+\De) = q^*(K_{Y'}+\De_{Y'})+E+F,
\] where $F \geq 0$ is a $p$-exceptional divisor and $E \geq 0$ is a $q$-exceptional divisor but not $p$-exceptional. By \eqref{eq: <}, $\Supp p(E) = \Exc(\nu^{-1}\circ g)$. Therefore, it suffices to show $E=0$.

Suppose that $E >0$ and $\Gamma$ is an irreducible component of $E$. As $K_{Y'}+\De_{Y'}$ is nef$/S$, 
\[
\sigma_{\Gamma}(q^*(K_{Y'}+\De_{Y'})+E+F; W/S) = \mult_{\Gamma} E>0,
\] where $\sigma_{\Gamma}(q^*(K_{Y'}+\De_{Y'})+E+F; W/S)$ is the coefficient of $\Gamma$ in the relative $\sigma$-decomposition of $q^*(K_{Y'}+\De_{Y'})+E+F$ (see \cite[Chapter III]{Nak04}). On the other hand, as $[K_X+\De]\in \bMov(X/S)$, if $\sigma_{\Gamma}(p^*(K_X+\De))>0$, then $\Gamma$ must be $p$-exceptional. This contradicts the choice of $\Gamma$.
\end{proof}

A weak log canonical model $(Y/S, \De_Y)$ of $(X/S, \De)$ is called a good minimal model of $(X/S, \De)$ if $K_Y+\De_Y$ is semi-ample$/S$. It is well-known that the existence of a good minimal model of $(X/S, \De)$ implies that any weak log canonical model of $(X/S, \De)$ is a good minimal model (for example, see the proof in \cite[Remark 2.7]{Bir12b}).

By saying that ``good minimal models of effective klt pairs exist in dimension $n$", we mean that for any projective variety $X$ of dimension $n$ over $\Cc$, if $(X, \De)$ is klt and the Kodaira dimension $\ka(K_X+\De) \geq 0$, then $(X, \De)$ has a good minimal model.

\begin{theorem}[{\cite[Theorem 2.12]{HX13}}]\label{thm: HX13}
Let $f: X \to S$ be a surjective projective morphism and $(X, \De)$ a klt pair such that for a very general closed point $s\in S$, the fiber $(X_s, \De_s=\De|_{X_s})$ has a good minimal model. Then $(X, \De)$ has a good minimal model over $S$.
\end{theorem}

\cite[Theorem 2.12]{HX13} states for a $\Qq$-divisor $\De$. However, it still holds for an $\Rr$-divisor $\De$: in the proof of \cite[Theorem 2.12]{HX13}, one only needs to replace $\proj_S \oplus_{m\in\Zz_{>0}}R^0f_*\Oo_X(m(K_X+\De))$ by the canonical model of $(X/S, \De)$ whose existence is known for effective klt pairs by \cite[Corollary 1.2]{Li22}. Indeed, because $\kappa(K_{X_s}+\De_s) \geq 0$ for a very general $s\in S$ by assumption, $K_X+\De \sim_\Rr E/S$ with $E \geq 0$ by \cite[Theorem 3.18]{Li22}.

\subsection{Shokurov polytopes}

Let $V$ be a finite-dimensional $\Rr$-vector space with a rational structure. A polytope (resp. rational polytope) $P\subset V$ is the convex hull of finite points (resp. rational points) in $V$. In particular, a polytope is always closed and bounded. We denote by $\Int(P)$ the relative interior of $P$, and refer to $\Int(P)$ as the relatively open polytope. By convention, a single point is a relatively open polytope.  Therefore, $\Rr_{>0}\cdot P$ is a relatively open polyhedral cone iff $P$ is a relatively open polytope.

\begin{theorem}[{\cite[Theorem 3.4]{SC11}}]\label{thm: Shokurov}
Let $X$ be a $\Qq$-factorial variety and $f: X \to S$ be a  fibration. Assume that good minimal models exist for effective klt pairs in dimension $\dim(X/S)$. Let $D_i, i=1,\ldots, k$ be effective $\Qq$-divisors on $X$. Suppose that $P \subset \oplus_{i=1}^k [0,1) D_i$ is a rational polytope such that for any $\De \in P$, $(X, \De)$ is klt and $\ka(K_F+ \De|_F)\geq 0$, where $F$ is a general fiber of $f$. 

Then $P$ can be decomposed into a disjoint union of finitely many relatively open rational polytopes $P = \sqcup_{i=1}^m Q^\circ_i$ such that for any $B, D \in Q^\circ_i$, if $(Y/S, B_Y)$ is a weak log canonical model of $(X/S, B)$, then $(Y/S, D_Y)$ is also a weak log canonical model of $(X/S, D)$.
\end{theorem}

For the convenience of the reader, we give the proof of Theorem \ref{thm: Shokurov}. The argument essentially follows from \cite[Lemma 7.1]{BCHM10}. However, we need to take care of the weaker assumption on the existence of weak log canonical models, as opposed to log terminal models. 

\begin{proof}[Proof of Theorem \ref{thm: Shokurov}]
We proceed by induction on the dimension of $P$. Note that by Theorem \ref{thm: HX13}, $(X/S, \De)$ has a good minimal model$/S$ for any $\De \in P$.

\noindent Step 1. If there exists a $\De_0 \in P$ such that $K_X+\De_0 \equiv 0/S$, then we show the claim. In fact, let $P'$ be the union of facets of $P$. By the induction hypothesis, $P'= \sqcup_j \ti Q^\circ_j$ is a finite union such that each $\ti Q^\circ_j$ is a relatively open rational polytope, and for $B, D \in \ti Q^\circ_j$, if $(Y/S, B_Y)$ is a weak log canonical model of $(X/S, B)$, then $(Y/S, D_Y)$ is also a weak log canonical model of $(X/S, D)$. Note that for any facet $F$ of $P$, each $\ti Q_j^\circ$ either lies entirely in $F$ or is disjoint from $F$. For  $t, t'\in(0,1]$, we have
\[
\begin{split}
&K_X+tB+(1-t)\De_0=t(K_X+B)+(1-t)(K_X+\De_0) \equiv t(K_X+B)/S,\\
&K_X+t'D+(1-t')\De_0=t'(K_X+D)+(1-t')(K_X+\De_0) \equiv t'(K_X+D)/S.
\end{split}
\] Hence, $(Y/S, tB_Y+(1-t)\De_{0,Y})$ is a weak log canonical model of $(X/S, tB+(1-t)\De_{0})$ iff $(Y/S, B_Y)$ is a weak log canonical model of $(X/S, B)$ iff $(Y/S, D_Y)$ is a weak log canonical model of $(X/S, D)$ iff $(Y/S, t'D_Y+(1-t')\De_{0,Y})$ is a weak log canonical model of $(X/S, t'D+(1-t')\De_{0})$. Therefore, if $\De_0 \in \Int(P)$, then the decomposition
\[
P=\left(\bigsqcup_j \ti Q_j^\circ\right) \bigsqcup \left(\bigsqcup_j \Int(\Conv(\ti Q_j^\circ, \De_0))\right) \bigsqcup \{\De_0\}
\] satisfies the claim. If $\De_0$ lies on the boundary of $P$, define
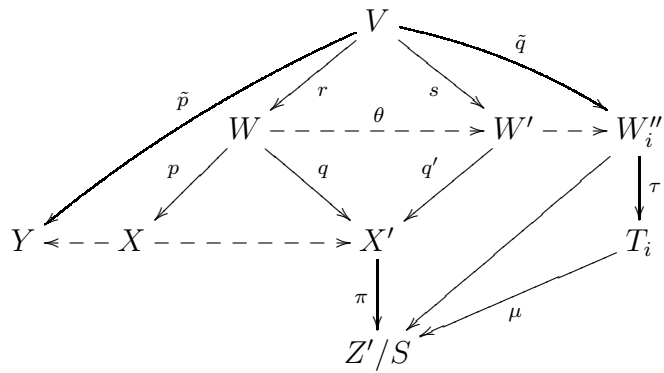
\[
P'' \coloneqq \bigcup_{\substack{\De_0\in F \\ F\subset P \text{~is a facet}}} F
\]
to be the union of facets of $P$ that contain $\De_0$.  Then the decomposition
\[
P=\left(\bigsqcup_{j} \ti Q_j^\circ\right) \bigsqcup \left(\bigsqcup_{\ti Q_j^\circ \not\subset P''} \Int(\Conv(\ti Q_j^\circ, \De_0))\right)
\] satisfies the claim. 

\[
\xymatrix{
&&&V\ar[ld]^r\ar[rd]_s\ar@/_/[llldd]_{\ti p}\ar@/^/[rrd]^{\ti q}&&\\
&&W\ar@{-->}[rr]^\theta\ar[ld]_p\ar[rd]^q&&W'\ar@{-->}[r]\ar[ld]_{q'}&W_i''\ar[lldd]\ar[d]^\tau\\
Y&X\ar@{-->}[l]\ar@{-->}[rr]&&X'\ar[d]_\pi&&T_i\ar[lld]^\mu\\
&&&Z'/S&&}
\]

\noindent Step 2. Next, we show the general case. By the compactness of $P$, it suffices to show the result locally around any point $\De_0\in P$. During the argument, by saying that shrinking $P$, we mean that replacing $P$ by a sufficiently small rational polytope $P'\subset P$ such that $P' \supset P \cap {\mathbb B}(\De_0,\ep)$, where  ${\mathbb B}(\De_0,\ep)\subset \oplus_{i=1}^k \Rr \cdot D_i$ is the ball centered at $\De_0$ with radius $\ep \in \Rr_{>0}$.

Let $(X'/S, \De_0')$ be a weak log canonical model of $(X/S, \De_0)$. By Theorem \ref{thm: HX13}, there exist a contraction $\pi: X' \to Z'/S$ and an ample$/S$ $\Rr$-Cartier divisor $A$ on $Z'$ such that $K_{X'}+\De_0' \sim_\Rr \pi^*A/S$. 

Let $p: W \to X, q: W \to X'$ be birational morphisms such that $q\circ p^{-1}$ is the natural map $X \dto X'$. Moreover, we assume that $p$ is a log resolution of $(X, \sum_{i=1}^k D_i)$. Let $\ti D_i, i=1,\ldots, k$ be the strict transforms of $D_i, i=1,\ldots, k$ on $W$, and $E_j, j=1, \ldots, l$ be prime $q$-exceptional divisors. Shrinking $P$, there exist some $0< \ep \ll 1$ and 
a linear map defined over $\Qq$, 
\[
L: P \to P_W, \quad \De \mapsto L(\De)\coloneqq \ti \De + (1-\ep) \Exc(p)
\] such that $P_W \subset (\oplus_{i=1}^k \ti D_i) \oplus (\oplus_{j=1}^l E_j)$ is a rational polytope, and
\begin{equation}\label{eq: 1}
K_W+L(\De)=p^*(K_X+\De)+E(\De),
\end{equation} such that $E(\De) \geq 0$ is $p$-exceptional and $(W, L(\De))$ is still klt for each $\De\in P$. Let $\De_{W,0} \coloneqq L(\De_0)$. Run a $(K_W+\De_{W,0})$-LMMP with scaling of an ample divisor over $X'$, then it terminates with $(W'/X', \De_{W',0})$ by \cite[Corollary 1.4.2]{BCHM10}. As $(X'/S, \De_0')$ is a weak log canonical model of $(X/S, \De_0)$, there exists a $q$-exceptional divisor $E_0 \geq 0$ such that
\[
p^*(K_X+\De_0)=q^*(K_{X'}+\De_0')+E_0.
\] Hence
\[
K_{W}+\De_{W,0}=q^*(K_{X'}+\De_0')+E(\De_0)+E_0.
\] Because $E(\De_0)+E_0 \geq 0$ is $q$-exceptional, we have
\[
K_{W'}+\De_{W',0}=q'^*(K_{X'}+\De_0'),
\] where $q': W' \to X'$ is the natural morphism. In particular,
\[
K_{W'}+\De_{W',0} \equiv 0/Z'.
\] Let $\theta: W \dto W'/X$ be the natural map. 
Shrinking $P$, we can assume that $\theta$ is $(K_W+L(\De))$-negative (see \cite[Definition 3.6.1]{BCHM10}) for each $\De\in P$. We set 
\[
L' \coloneqq \theta_*\circ L \text{~and~} P_{W'} \coloneqq L'(P)=\theta_*P_W.
\]

\noindent Step 3. By Step 1, $P_{W'} = \sqcup Q'^\circ_i$ can be decomposed into a disjoint union of finitely many relatively open rational polytopes such that for any $B',D'\in Q'^\circ_i$, if $(W_i''/Z', B'')$ is a weak log canonical model of $(W'/Z', B')$ then $(W_i''/Z', D'')$ is a weak log canonical model of $(W'/Z', D')$, where $B'',D''$ are the strict transforms of $B',D'$ respectively. In the sequel, we fix a $W''_i$ for each $Q'^\circ_i$.

We claim that after shrinking $P_{W'}$, for any $\De_{i} \in Q'^\circ_i$, $K_{W_i''}+\De_{i}''$ is nef over $S$, where $\De_{i}''$ is the strict transform of $\De_{i}$. Let $\De_{i}$ be a vertex of $\bar Q'^\circ_i$. By Theorem \ref{thm: HX13}, $(W_i''/Z', \De_{i}'')$ is semi-ample$/Z'$. Let $\tau: W_i'' \to T_i/Z'$ be the morphism such that $K_{W_i''}+\De_{i}''\sim_\Qq \tau^*H_i/Z'$, where $H_i$ is an ample$/Z'$ $\Qq$-Cartier divisor on $T_i$. Hence, there is a $\Qq$-Cartier divisor $\Theta$ on $Z'$ such that $K_{W_i''}+\De_i''=\tau^*H_i+(\mu\circ\tau)^*\Theta$, where $\mu: T_i \to Z'$. Then $t(H_i+\mu^*\Theta) +(1-t)\mu^*A$ is nef over $S$ when $t \in [0,t_0]$ for some rational number $0<t_0\ll 1$. Note that
\[
t(K_{W_i''}+\De_{i}'')+(1-t)(K_{W_i''}+\De_{W,0}'')\sim_\Rr \tau^*\left(t(H_i+\mu^*\Theta)+(1-t)\mu^*A\right)/S,
\] where $\De_{W,0}''$ is the strict transform of $\De_{W,0}$ on $W''_i$. Replacing $\De_{i}$ by $t_0 \De_{i}+(1-t_0)\De_{W,0}$ and repeating this process for each vertex of $\bar Q'^\circ_i$, we obtain a polytope satisfying the desired claim. Moreover, this property of $P_{W'}$ holds for any weak log canonical model $(W_i'''/Z', \De_i''')$ of $(W'/Z', \De_i)$. Indeed, if $\mu: V \to W_i'''$ and $\ti q: V \to W_i''$ are birational morphisms such that $\ti q \circ \mu^{-1}$ is the natural map $W_i''' \dto W_i''$, then the negativity lemma (see \cite[Lemma 3.39]{KM98}) implies that $\mu^*(K_{W_i'''}+\De_i''')=\ti q^*(K_{W_i''}+\De_i'')$. Hence, $K_{W_i'''}+\De_i'''$ is nef over $S$.

\noindent Step 4. Let $P\coloneqq (L')^{-1}(P_W)$ be the polytope corresponding to $P_W$ under the map $L'$. Let $Q^\circ_i \coloneqq (L')^{-1}(Q'^\circ_i)$, which is not necessarily a relatively open polytope. However, $Q^\circ_i$ can be written as a disjoint union of finitely many relatively open rational polytopes as $L'$ is a linear map defined over $\Qq$. Hence, without loss of generality, we can assume that $Q^\circ_i$ is a relatively open rational polytope. 

Therefore, we have
\[
P = \sqcup_i Q^\circ_i,
\] which is a disjoint union of finitely many relatively open rational polytopes. To complete the proof, it suffices to show that for any $B, D \in Q^\circ_i$, if $(Y/S, B_Y)$ is a weak log canonical model of $(X/S, B)$ then $(Y/S, D_Y)$ is a weak log canonical model of $(X/S, D)$. 

Let $r: V \to W, s: V \to W', \ti p: V \to Y, \ti q: V \to W_i''$ be birational morphisms which commute with the existing maps. Moreover, $r, s, \ti p, \ti q$ can be assumed to be log resolutions. By \eqref{eq: 1}, for $\De \in Q_i^\circ$,
\begin{equation}\label{eq:2}
r^*(K_W+L(\De))=r^*p^*(K_X+\De)+r^*(E(\De)).
\end{equation} As $\theta: W \dto W'$ is $(K_W+L(\De))$-negative, $(W_i''/S, L(\De)'')$ is also a weak log canonical model of $(W/S, L(\De))$, where $L(\De)''$ is the strict transform of $L(\De)$. Then there is a $\ti q$-exceptional divisor $F(\De) \geq 0$ such that
\[
r^*(K_W+L(\De)) = \ti q^*(K_{W_i''}+L(\De)'')+F(\De).
\] Combining with \eqref{eq:2}, we have
\[
\ti q^*(K_{W_i''}+L(\De)'')+F(\De)=r^* p^*(K_X+\De)+r^*(E(\De)).
\] Hence $-F(\De)+r^*(E(\De))$ is nef over $X$. As
\[
(p\circ r)_*(-F(\De)+r^*(E(\De)))=(p\circ r)_*(-F(\De)) \leq 0,
\] we have
\[
-F(\De)+r^*(E(\De)) \leq 0
\] by the negativity lemma. Let
\begin{equation}\label{eq:star}
r^*p^*(K_X+\De)=\ti p^*(K_Y+\De_Y)+\Theta(\De),
\end{equation} where $\Theta(\De)$ is $\ti p$-exceptional. Hence
\[
\ti q^*(K_{W_i''}+L(\De)'')+F(\De)=\ti p^*(K_Y+\De_Y) + \Theta(\De)+r^*(E(\De)).
\] As $-F(\De) + \Theta(\De)+r^*(E(\De))$ is nef over $Y$ and $$\ti p_*(-F(\De) + \Theta(\De)+r^*(E(\De)))=\ti p_*(-F(\De)) \leq 0,$$ we have $-F(\De) + \Theta(\De)+r^*(E(\De)) \leq 0$ by the negativity lemma. 

Now we use that $(Y/S, B_Y)$ is a weak log canonical model of $(X/S, B)$. As $K_Y+B_Y$ is nef$/S$, $F(B) - \Theta(B)-r^*(E(B))$ is nef over $W_i''$. As $F(B)$ is $\ti q$-exceptional, we have
\[
\ti q_*(F(B) - \Theta(B)-r^*(E(B))) \leq 0.
\] By the negativity lemma again, we have $F(B) - \Theta(B)-r^*(E(B)) \leq 0$. Therefore, $F(B) - \Theta(B)-r^*(E(B)) = 0$. Note that $F(\De) - \Theta(\De)-r^*(E(\De))$ is linear for $\De \in Q^\circ_i$. Because $Q^\circ_i$ is relatively open and $F(\De) - \Theta(\De)-r^*(E(\De)) \geq 0$ for each $\De \in Q^\circ_i$, we have $F(\De) - \Theta(\De)-r^*(E(\De)) = 0$ for each $\De \in Q^\circ_i$. Thus $\Theta(\De)=F(\De)-r^*(E(\De)) \geq 0$ and
\[
\ti q^*(K_{W_i''}+L(\De)'') = \ti p^*(K_Y+\De_Y)
\] is nef$/S$ for any $\De \in Q^\circ_i$. As $X \dto Y$ does not extract divisors, $(Y/S, D_Y)$ is also a weak log canonical model of $(X/S, D)$ by \eqref{eq:star}. 
\end{proof}

\begin{remark}\label{rmk: use good mm}
In the proof of Theorem \ref{thm: Shokurov}, we use the existence of good minimal models in Step 1 and Step 3. In Step 1, this is needed to ensure that the statement of Theorem \ref{thm: Shokurov} holds true for lower-dimensional polytopes. In Step 3, let $L'(\De)$ be the strict transform of $L(\De)$ on $W'$, then we need that $(F', L'(\De)|_{F'})$ has a good minimal model, where $F'$ is a general fiber of $\pi\circ q': W' \to Z'$.
\end{remark}

\begin{theorem}\label{thm: Shokurov-Choi}
Let $(X,\De)\to S$ be a klt Calabi-Yau fiber space. Assume that good minimal models of effective klt pairs exist in dimension $\dim(X/S)$. Let $P \subset \Eff(X/S)$ be a rational polyhedral cone. Then $P$ is a finite union of relatively open rational polyhedral cones $P = \cup_{i=0}^m P^\circ_i$ such that whenever
\begin{enumerate}
\item $B, D$ are effective divisors with $[B], [D] \in P^\circ_i$, and
\item $(X, \De+\ep B), (X, \De+\ep D)$ are klt for some $\ep \in \Rr_{> 0}$,
\end{enumerate} 
then if $(Y/S, \De_Y+\ep B_Y)$ is a weak log canonical model of $(X/S, \De+\ep B)$, then $(Y/S, \De_Y+\ep D_Y)$ is a weak log canonical model of $(X/S, \De+\ep D)$.
\end{theorem}
\begin{proof}
Let $\De=\sum_{i=1}^m c_i \De_i$ be the decomposition into irreducible components. Then
\[
\{\Theta \in \oplus_{i=1}^m \Rr \cdot \De_i \mid K_X+\Theta \equiv 0/S\} 
\] is a subspace of $\oplus_{i=1}^m \Rr \cdot \De_i$ defined over $\Qq$. Hence, there exists a $\Qq$-Cartier divisor $\ti \De$ such that $(X, \ti\De) \to S$ is a klt Calabi-Yau fiber space. Let $\ti\ep \in \Rr_{>0}$ such that $(X, \ti\De+\ti\ep B)$ is klt. By $$K_X+\De+\ep B \equiv \frac{\ep}{\ti \ep}(K_X+\ti\De+\ti\ep B)/S,$$ $(Y/S, \De_Y+\ep B_Y)$ is a weak log canonical model of $(X/S,\De+\ep B)$ iff $(Y/S, \ti\De_Y+\ti\ep B_Y)$ is a weak log canonical model of $(X/S,\ti\De+\ti\ep B)$. Therefore, replacing $\De$ by $\ti\De$, we can assume that $\De$ is a $\Qq$-Cartier divisor.

Let $\ti P = \Conv(\De_j \mid j=1\ldots k)$ be a rational polytope generated by effective $\Qq$-Cartier divisors $\De_j \geq 0, j=1\ldots k$ such that $\Rr_{\geq 0} \cdot [\ti P] = P$. Here $[\ti P]$ is the image of $\ti P$ in $N^1(X/S)_\Rr$. We can choose $\ti P$ such that $0 \not\in [\ti P]$. Replacing $\De_j$ by $\ep \De_j$ for some $\ep \in \Qq_{>0}$, we can assume that $(X, \De+\De_j)$ is klt for each $j$. Let $$\ti P+\De = \sqcup_{i=1}^m (Q_i^\circ+\De)$$ be the decomposition as in Theorem \ref{thm: Shokurov}. For each relatively open rational polytope $Q_i^\circ+\De$, and $\Theta_1+\De, \Theta_2+\De \in Q_i^\circ+\De$, if $(Y/S, \Theta_{1,Y}+\De_Y)$ is a weak log canonical model of $(X/S, \Theta_1+\De)$, then $(Y/S, \Theta_{2,Y}+\De_Y)$ is a weak log canonical model of $(X/S, \Theta_2+\De)$. Let $P_i^\circ$ be the image of the relatively open rational polyhedral cone $\Rr_{>0} \cdot Q_i^\circ$ in $N^1(X/S)_\Rr$. Set $P_0^\circ=\{0\}$, then $P = \cup_{i=0}^m P_i^\circ$. Note that this union may not be disjoint.

The claim certainly holds true for $P_0^\circ$. For effective divisors $B, D$ with $[B],[D]\in P_i^\circ, i>0$, there exist $\De_B, \De_D \in Q_i^\circ$ such that
\[
[B]=t[\De_B],\quad [D]=s[\De_D] \quad \text{for some~} t, s \in \Rr_{>0}.
\] By $K_X+\De \equiv 0/S$,
\begin{equation}\label{eq:num propotion}
\begin{split}
K_X+\De+\De_B &\equiv \frac{1}{\ep t}(K_X+\De+\ep B)/S\\
K_X+\De+\De_D &\equiv \frac{1}{\ep s}(K_X+\De+\ep D)/S.
\end{split}
\end{equation} Therefore, $(Y/S, \De_Y+\ep B_Y)$ is a weak log canonical model of $(X/S, \De+\ep B)$ iff $(Y/S, \De_Y+\De_{B,Y})$ is a weak log canonical model of $(X/S, \De+\De_B)$. By Theorem \ref{thm: Shokurov}, this implies that $(Y/S, \De_Y+\De_{D,Y})$ is a weak log canonical model of $(X/S, \De+\De_D)$. Hence $(Y/S, \De_Y+\ep D_Y)$ is a weak log canonical model of $(X/S, \De+\ep D)$ by \eqref{eq:num propotion} again.
\end{proof}

\begin{theorem}[{\cite[\S 6.2. First main theorem]{Sho96}}]\label{thm: nef cone is polyhedral}
Let $(X, \De) \to S$ be a klt Calabi-Yau fiber space. Let $P\subset \Eff(X/S)$ be a rational polyhedral cone. Then
\[
P_N \coloneqq P \cap \bAmp(X/S)
\] is a rational polyhedral cone.
\end{theorem}
\begin{proof}
Let $D_j, 1 \leq j \leq k$ be effective $\Qq$-Cartier divisors on $X$ such that $P = \Cone([D_j] \mid 1 \leq j \leq k)$. Replacing $D_j$ by $\ep D_j$ for some $\ep \in \Qq_{>0}$, we can assume that $(X, \De+ D_j)$ is klt for each $j$. Then
\[
\NN = \{D \in \oplus_{j=1}^k [0,1]D_j \mid D \text{~is nef over~} S\}
\] is a rational polytope by \cite[\S 6.2. First main theorem]{Sho96} (also see \cite[Proposition 3.2 (3)]{Bir11}). The image $[\NN]$ of $\NN$ in $N^1(X/S)_\Rr$ is still a rational polytope. By the construction,
\[
P_N = \Cone([\NN]).
\] Thus $P_N$ is a rational polyhedral cone.
\end{proof}

\section{Geometry of convex cones}\label{sec: Geometry of convex cones}

Let $V(\Zz)$ be a lattice and $V(\Qq) \coloneqq V(\Zz) \otimes_\Zz \Qq$, $V  \coloneqq V(\Qq) \otimes_\Qq \Rr$. A cone $C \subset V$ is non-degenerate if it does not contain an affine line. This is equivalent to saying that its closure $\bar C$ does not contain a nontrivial vector subspace.

In the following, we assume that $\Gamma$ is a group and $\rho: \Gamma \to {\rm GL(V)}$ is a group homomorphism. The group $\Gamma$ acts on $V$ through $\rho$. For $\gamma \in \Gamma$ and $x \in V$, we write $\gamma \cdot x$ or $\gamma x$ for the action. For a set $S\subset V$, set $\Gamma \cdot S \coloneqq \{\gamma \cdot x \mid  \gamma \in \Gamma, x\in S\}$. Suppose that this action leaves a convex cone $C$ and some lattice in $V(\Qq)$ invariant. We assume that $\dim C=\dim V$. The following definition slightly generalizes \cite[
Proposition-Definition 4.1]{Loo14}.

\begin{definition}\label{def: polyhedral type} 
Under the above notation and assumptions.
\begin{enumerate}
\item Suppose that $C \subset V$ is an open convex cone (may be degenerate). Let
\[
C_+ \coloneqq \Conv(\bar C \cap V(\Qq))
\] be the convex hull of rational points in $\bar C$. 

\item We say that $(C_+, \Gamma)$ is of polyhedral type if there is a polyhedral cone $\Pi \subset C_+$ such that $\Gamma \cdot \Pi \supset C$.
\end{enumerate}
\end{definition}

\begin{remark}
Recall that a polyhedral cone is closed by definition (see Section \ref{subsec: Movable cones and ample cones}). Also note that the openness of $C$ is not strictly necessary (see \cite[Definition~3.2]{GLSW24}). For instance, Definition \ref{def: polyhedral type} (2) could be modified to require $\Gamma \cdot \Pi \supset \Int(C)$. However, we adopt only a minimal modification of the original definition in \cite{Loo14}.
\end{remark}

\begin{proposition}[{\cite[Proposition-Definition 4.1]{Loo14}}]\label{prop: prop-def}
Under the above notation and assumptions. If $C$ is non-degenerate, then the following conditions are equivalent:
\begin{enumerate}
\item there exists a polyhedral cone $\Pi \subset C_+$ with $\Gamma \cdot  \Pi = C_+$;
\item there exists a polyhedral cone $\Pi \subset C_+$ with $\Gamma \cdot  \Pi \supset C$.
\end{enumerate}
Moreover, in case (2), we necessarily have $\Gamma \cdot  \Pi = C_+$.
\end{proposition}

\begin{definition}\label{def: fundamental domain}
Let $\rho: \Gamma \hookrightarrow {\rm GL}(V)$ be an injective group homomorphism and $C \subset V$ be a cone (may not necessarily be open). Let $\Pi \subset C $ be a (rational) polyhedral cone. Suppose that $\Gamma$ acts on $C$. Then $\Pi$ is called a weak (rational) polyhedral fundamental domain for $C$ under the action $\Gamma$ if 
\begin{enumerate}
\item $\Gamma \cdot \Pi = C$, and
\item for each $\gamma \in \Gamma$, either $\gamma \Pi = \Pi$ or $\gamma\Pi \cap \Int(\Pi) = \emptyset$. 
\end{enumerate}

Moreover, for $\Gamma_\Pi\coloneqq \{\gamma \in \Gamma \mid \gamma\Pi = \Pi\}$, if $\Gamma_\Pi=\{{\rm id}\}$, then $\Pi$ is called a (rational) polyhedral fundamental domain.
\end{definition}

\begin{lemma}[{\cite[Theorem 3.8 \& Application 4.14]{Loo14}}]\label{le: existence of fun domain}
Under the notation and assumptions of Definition \ref{def: polyhedral type}, suppose that $\rho: \Gamma \hookrightarrow {\rm GL}(V)$ is injective. Let $(C_+, \Gamma)$ be of polyhedral type with $C$ non-degenerate. Then under the action of $\Gamma$, $C_+$ admits a rational polyhedral fundamental domain.
\end{lemma}
\begin{proof}
Let $V^*$ be the dual vector space of $V$ with pairing $$\la -,-\ra: V \times V^* \to \Rr.$$ Let
\[
C^* \coloneqq \{y \in V^* \mid \la x, y \ra \geq 0 \text{~for all~} x \in C\}
\] be the dual cone of $C$, and ${\rm Int}(C^{*})$ be the relative interior of $C^*$. By $C$ non-degenerate and $\dim C= \dim V$, we still have $\dim  {\rm Int}(C^{*}) = \dim V$.

The group $\Gamma$ naturally acts on $V^*$. In fact, for $\gamma \in \Gamma$ and a $y\in V^*$, $\gamma \cdot y$ is defined by the relation $\la x, \gamma\cdot y \ra = \la \gamma\cdot x, y \ra$ for all $x \in V$. It is straightforward to check that this action gives an injective group homomorphism $\Gamma \hookrightarrow {\rm GL}(V^*)$ which leaves $C^*$ and a lattice in $V^*(\Qq)$ invariant. Therefore, by \cite[Theorem 3.8]{Loo14}, $\Gamma$ acts properly discontinuously  on ${\rm Int}(C^{*})$.

By \cite[Application 4.14]{Loo14}, for each $\xi \in {\rm Int}(C^{*}) \cap V(\Qq)^*$, there is a rational polyhedral cone $\sigma$ associated with $\xi$, such that $\sigma$ is a rational polyhedral fundamental domain for the action of $\Gamma$ on $C_+$ whenever the stabilizer subgroup $\Gamma_\xi = \{1\}$. It suffices to find such $\xi$ to complete the proof. As $\Gamma$ acts properly discontinuously on ${\rm Int}(C^{*})$, for any polyhedral cone $P \subset {\rm Int}(C^{*})$ such that $\dim  P= \dim {\rm Int}(C^{*})=\dim V$, the set $$\{\gamma \in \Gamma \mid \gamma P^\circ \cap P^\circ \neq \emptyset\}$$ is a finite set.
 Then a general $\xi \in P^\circ \cap V^*(\Qq)$ satisfies $\Gamma_{\xi}=\{1\}$.
\end{proof}

The following consequence of having a polyhedral fundamental domain is well-known (see \cite[Corollary 4.15]{Loo14} or \cite[(4.7.7) Proposition]{Mor15})

\begin{theorem}\label{thm: finite presented}
Let $\rho: \Gamma \hookrightarrow {\rm GL}(V)$ be an injective group homomorphism and $C \subset V$ be a cone. Suppose that $C$ is $\Gamma$-invariant. If $C$ admits a polyhedral fundamental domain under the action of $\Gamma$, then $\Gamma$ is finitely presented. 
\end{theorem}

For a possibly degenerate open convex cone $C$, let $W \subset \bar C$ be the maximal $\Rr$-linear vector space. We say that $W$ is defined over $\Qq$ if $W = W(\Qq) \otimes_\Qq \Rr$ where $W(\Qq) = W \cap V(\Qq)$. In this case, $V/W=(V(\Qq)/W(\Qq)) \otimes_\Qq \Rr$ has a natural lattice structure, and we denote everything in $V/W$ by $\ti{(-)}$. For example, $\widetilde{(C_+)}$ is the image of $C_+$ under the projection  $p:V \to V/W$. By the maximality, $W$ is $\Gamma$-invariant, and thus $V/W, \ti C$ admit natural $\Gamma$-actions.

\begin{lemma}\label{le: induced polyhedral type}
Under the above notation and assumptions, 
\begin{enumerate}
\item $\bar{\ti C} = \ti{\bar C}$,
\item $({\ti C})_+ = \widetilde{(C_+)}$, which is denoted by $\ti C_+$, and
\item if $(C_+, \Gamma)$ is of polyhedral type, then $(\ti C_+, \Gamma)$ is still of polyhedral type. More precisely, if $\Pi \subset C_+$ is a polyhedral cone with $\Gamma \cdot \Pi \supset C$, then $\ti \Pi \subset \ti C_+$ and $\Gamma \cdot \ti\Pi \supset \ti C$. 
\end{enumerate}
\end{lemma} 

\begin{proof}
For (1), $\bar{\ti C} \supset \ti{\bar C}$ trivially holds. The converse does not hold for an arbitrary linear projection (see \cite[Remark 2.5]{Loo14}). In our case, let $j: V/W \to V$ be a splitting of $p: V \to V/W$. For $x \in \bar{\ti C}$, let $\ti a_n \to x$ with $a_n \in C$. Then $a_n - j(\ti a_n) \in W$ as $p(a_n - j(\ti a_n))=0$. By $W \subset \bar C$, $j(\ti a_n) \in \bar C$. Moreover, as $\{\ti a_n\}_{n\in \Nn}$ converges and $j$ is continuous, $\{j(\ti a_n)\}_{n\in \Nn}$ converges to $\alpha \in \bar C$. Thus $\ti a_n = p(j(\ti a_n)) \to p(\alpha)$. Hence $x = p(\alpha) \in \ti{\bar C}$.

For (2), as $W$ is defined over $\Qq$, $V/W$ has the natural $\Qq$-structure $V/W(\Qq)\coloneqq V(\Qq)/W(\Qq)$. We first show 
\[
\ti{\bar C} \cap (V/W)(\Qq) = \widetilde{(\bar C \cap V(\Qq))}.
\] The $``\supset"$ follows from definition. For the converse, let $\ti a \in \ti{\bar C}$ be a rational point in $V/W$. Then by $W \subset \bar C$, we can assume that $a \in \bar C$ is a rational point in $V$. This gives $``\subset"$. Next, we show
\[
\Conv\left(\widetilde{\bar C \cap V(\Qq)}\right) = {\widetilde\Conv(\bar C \cap V(\Qq))}.
\] As the image of a convex set is still convex, we have $``\subset"$. 

For a set $S \subset V$, we have
\[
\Conv(S) = \{\sum_{i \in I} \lambda_i s_i \mid \sum_{i\in I} \lambda_i =1, \lambda_i >0, |I|<\infty, s_i \in S\}.
\] For $a \in {\Conv(\bar C \cap V(\Qq))}$, take finitely many $s_i \in \bar C \cap V(\Qq)$, and $\lambda_i >0, \sum_{i\in I} \lambda_i =1$, so that $a=\sum \lambda_i s_i$. Thus $\ti a = \sum \lambda_i \ti s_i \in \Conv\left(\widetilde{\bar C \cap V(\Qq)}\right) $. This shows the converse inclusion. 

Finally, by (1), 
\[
({\ti C})_+ = \Conv(\bar{\ti C} \cap (V/W)(\Qq)) =  \Conv(\ti{\bar C} \cap (V/W)(\Qq)). 
\] Then (2) follows from
\[
\Conv(\ti{\bar C} \cap (V/W)(\Qq)) = \Conv\left(\widetilde{\bar C \cap V(\Qq)}\right) = {\widetilde\Conv(\bar C \cap V(\Qq))} = \widetilde{(C_+)}.
\] 

For (3),  let $\Pi \subset C_+$ be a polyhedral cone such that $\Gamma \cdot \Pi \supset C$. By (2), $\ti \Pi \subset \ti C_+$. Moreover, $\Gamma \cdot \ti\Pi \supset \ti C$. 
\end{proof} 

The following proposition generalizes \cite[Theorem 3.8 and Application 4.14]{Loo14} (see Lemma \ref{le: existence of fun domain}) to the degenerate case.

\begin{proposition}\label{prop: degenerate cone}
Let $(C_+, \Gamma)$ be of polyhedral type. Let $W \subset \bar C$ be the maximal vector space. Suppose that $W$ is defined over $\Qq$. Then there is a rational polyhedral cone $\Pi \subset C_+$ such that $\Gamma \cdot \Pi = C_+$, and for each $\gamma \in \Gamma$, either $\gamma \Pi\cap \Int(\Pi) = \emptyset$ or $\gamma \Pi = \Pi$. Moreover, 
\[
\{\gamma \in \Gamma \mid \gamma \Pi = \Pi\} = \{\gamma \in \Gamma \mid \gamma \text{~acts trivially on ~}V/W\}.
\]
\end{proposition}
\begin{proof}
By Lemma \ref{le: induced polyhedral type} (3), $(\ti C_+, \Gamma)$ is still of polyhedral type. By Lemma \ref{le: existence of fun domain}, there is a rational polyhedral cone $\ti \Pi$ as a fundamental domain of $\ti C_+$ under the action of $\ti \rho(\Gamma)$, where $\ti\rho: \Gamma \to {\rm GL}(V/W)$ is the natural group homomorphism. By Lemma \ref{le: induced polyhedral type} (2), let $\Pi' \subset C_+$ be a  rational polyhedral cone such that $p(\Pi')=\ti\Pi$, where $p: V \to V/W$. Let $\Pi \coloneqq  \Pi'+W$ which is a rational polyhedral cone. As $\gamma (\Pi'+W) = (\gamma \Pi')+W$, by Lemma \ref{le: induced polyhedral type} (2), we have $\Gamma \cdot \Pi = C_+$. 

If $\gamma \ti\Pi \cap \Int(\ti\Pi) = \emptyset$, then $\gamma \Pi \cap \Int(\Pi) = \emptyset$ as $\Int(\Pi)$ maps to $\Int(\ti\Pi)$. If $\gamma \ti\Pi =\ti\Pi$, then we claim that $\gamma \Pi = \Pi$. In fact, for some $a\in \Pi'$, we have $\widetilde{(\gamma  \cdot  a)}=\gamma  \cdot \ti a \in \ti \Pi$ and thus $\gamma  \cdot  a = b+w$ for some $b\in \Pi', w\in W$. Thus $\gamma  \Pi \subset \Pi$. Similarly, $\gamma^{-1}  \Pi \subset \Pi$. This shows the claim. Moreover, $\gamma \Pi = \Pi$ iff $\gamma$ acts trivially on $\ti \Pi$ iff $\gamma$ acts trivially on $V/W$ because $\ti \Pi$ is a fundamental domain under the action of $\ti\rho(\Gamma)$.
\end{proof}

\section{Generic properties of fibrations and structures of cones}\label{sec: Generic properties of fibrations and structures of cones}

Let $(X, \De) \to S$ be a fiber space. Let $K\coloneqq K(S)$ be the field of rational functions on $S$ and $\bar K$ be the algebraic closure of $K$. Then $X_{\bar K}\coloneqq X \times_S \spec\bar K$ is the geometric fiber of $f$. Set $\De_{\bar K} \coloneqq \De \times_S \spec\bar K$.

\begin{proposition}\label{prop: Generic property 0}
Let $f: X \to S$ be a fibration.

\item \begin{enumerate}
\item If $(X, \De)$ has klt singularities, then $(X_{\bar K}, \De_{\bar K})$ still has klt singularities. Moroever, if $f: (X, \De) \to S$ is a klt Calabi-Yau fiber space, then $(X_{\bar K}, \De_{\bar K})$ is a klt Calabi-Yau pair over $\spec \bar K$.

\item For a finite base change $h: T \to S$ between varieties, let $U \subset S$ be a non-empty open set and $V=h^{-1}(U)$. Then we can shrink $U$ such that $X_V \coloneqq X \times_S V$ satisfies the following properties.

If $(X, \De)$ has klt singularities, then $(X_V, \De_V)$ still has klt singularities, where $\De_V \coloneqq \De \times_S V$. Moreover, if $f: (X, \De) \to S$ is a klt Calabi-Yau fiber space, then $(X_V, \De_V)$ has klt singularities and $K_{X_V}+\De_{V} \sim_\Rr 0/V$.

\end{enumerate}
\end{proposition}
\begin{proof}
For (1), we first show that $X_{\bar K}$ is normal. This is a local statement for both source and target, hence we can assume that $f: \spec A \to \spec B$. The collection of affine open sets $\{\spec B_i \subset \spec B\mid i\}$ forms a direct system such that $\varinjlim B_i = K$. Then $A \otimes_B \varinjlim B_i= \varinjlim (A \otimes_B  B_i)$. As $A \otimes_B  B_i$ is normal, by \cite[\href{https://stacks.math.columbia.edu/tag/037D}{Lemma 037D}]{Sta22}, $A \otimes_B K$ is also normal. Then $X_{\bar K} = X_{K} \otimes_{K} \bar K$ is normal by \cite[\href{https://stacks.math.columbia.edu/tag/0C3M}{Lemma 0C3M}]{Sta22}. Let $v: \spec \bar K \to S, u: X_{\bar K} \to X$ and $\bar f: X_{\bar K} \to \spec \bar K$ be natural morphisms. Then $\bar f_*(u^*\Oo_X)=v^*(f_*\Oo_X)$. By $u^*\Oo_X=\Oo_{X_{\bar K}}$ and $f_*\Oo_{X}=\Oo_S$, we see that $X_{\bar K}$ is connected. Hence $X_{\bar K}$ is an irreducible normal variety over $\bar K$.

Next, we show that $X_{\bar K}$ has klt singularities. Let $X_{\reg}$ be the smooth part of $X$. Shrinking $S$, we can assume that $S$ is smooth and $X_{\reg} \to S$ is a smooth morphism. Then the sequence 
\[
0\to f^*\Omega_S \to \Omega_{X_{\reg}} \to \Omega_{X_{\reg}/S} \to 0
\] is exact. Let $r= \dim(X/S)$. By
\[
(\Omega_{X_{\reg}/S})_{\bar K} = \Omega_{(X_{\reg})_{\bar K}} \text{~and~} \Oo_X(K_{X_{\reg}/S}) = \wedge^r \Omega_{X_{\reg}/S},
\]we have
\[
(K_{X_{\reg}/S})_{\bar K} \sim K_{(X_{\reg})_{\bar K}}.
\] By $\codim(X\backslash X_{\rm reg}) \geq 2$, we have
\[
(K_{X/S})_{\bar K} - K_{X_{\bar K}}\sim 0.
\]Take a log resolution $g: \ti X \to X$, then
\[
K_{\ti X/S}+\ti \De = g^*(K_{X/S}+\De) 
\] with coefficients of $\ti \De<1$. The natural morphism $\bar g: \ti X_{\bar K} \to X_{\bar K}$ is also a log resolution and the above argument implies that 
\[
K_{\ti X_{\bar K}} +\ti \De_{\bar K}= \bar g^*(K_{X_{\bar K}}+\De_{\bar K}).
\] As coefficients of $\ti \De_{\bar K}<1$, $(X_{\bar K}, \De_{\bar K})$ has klt singularities. 

When $f: (X, \De) \to S$ is a klt Calabi-Yau fiber space. We only need to note that $K_{X}+\De \sim_\Rr 0/S$ implies that $K_{X_{\bar K}}+\De_{\bar K} \sim_{\Rr} 0$.

For (2), shrinking $U$, we can assume that $V \to U$ is \'etale. We first show that $X_V$ is normal. Note that $\phi: X_V \to X_U$ is also \'etale, where $X_U \coloneqq X \times_S U$. Let $x\in X_V$ be a point (not necessarily a closed point) and $y=\phi(x)$. Set $\Oo_{x}\coloneqq\Oo_{X_V,x}$ (resp. $\Oo_{y}\coloneqq\Oo_{X_U,y}$). Let $\hat \Oo_x$ (resp. $\hat\Oo_y$) be the completion with respect to the maximal ideal. By \cite[III, Exercise 10.4]{Har77},
\[
\hat\Oo_{y}\otimes _{k(y)}k(x) \simeq \hat\Oo_{x},
\] where $k(y)\subset \hat\Oo_{y}$ and $k(x) \subset \hat\Oo_{x}$ are fields of representatives. Note that $k(y)$ and $k(x)$ are of characteristic zero. We claim that  $\hat\Oo_{x}$ is normal. In fact, as $X$ has klt singularities, $X$ is Cohen-Macaulay. Hence $\hat\Oo_{y}$ is Cohen-Macaulay by \cite[\href{https://stacks.math.columbia.edu/tag/07NX}{Lemma 07NX}]{Sta22}. In particular, it satisfies Serre's condition $S_2$. As $\hat\Oo_{y}$ is certainly regular in codimension $1$, it is normal. Then \cite[\href{https://stacks.math.columbia.edu/tag/0C3M}{Lemma 0C3M}]{Sta22} shows that $\hat\Oo_{y}\otimes _{k(y)}k(x)$ is normal. Thus $\Oo_{x}$ is normal by \cite[\href{https://stacks.math.columbia.edu/tag/0FIZ}{Lemma 0FIZ}]{Sta22}. This shows that $X_V$ is normal.  

Let $f_V: X_V \to V$ be the natural map. By $V \to U$ flat and $f_*\Oo_X=\Oo_S$, we have $(f_V)_*\Oo_{X_V} =\Oo_V$. This implies $H^0(X_V, \Oo_{X_V})=H^0(V,\Oo_V)$ is an integral domain. This shows that $X_V$ is an irreducible normal variety. 

Let $\pi: W \to X$ be a log resolution of $(X, \De)$ with natural morphisms $\pi_V: W_V \to X_V$ and $\ti \phi: W_V \to W_U$. Set 
\[
K_W+\ti\De\coloneqq\pi^*(K_X+\De) \text{~and~} K_{W_V}+D\coloneqq\pi_V^*(K_{X_V}+\De_V).
\] As $\phi$ is \'etale, we have 
\[
K_{X_V}+\De_{V}=\phi^*(K_{X_U}+\De_U).
\] Therefore, we have $K_{W_V}+D=\ti \phi^*(K_{W_U}+\ti\De_U)$, where $\ti \De_U \coloneqq \ti\De \times_S U$. By $(X, \De)$ klt, coefficients of $\ti\De$ are $<1$. As $\ti \phi$ is \'etale, we have $D=\ti \phi^*\ti\De_{U}$, and thus the coefficients of $D$ are $<1$. As $W_V \to X_V$ is a log resolution of $(X_V, \De_V)$, $(X_V, \De_V)$ is still klt. 

When $f: (X, \De) \to S$ is a klt Calabi-Yau fiber space, then $K_X+\De \sim_\Rr 0/S$ implies that $K_{X_V}+\De_V=\phi^*(K_X+\De) \sim_\Rr 0/V$.
\end{proof}

In the sequel, we will use the following spreading-out and specialization techniques, whose spirit is well known (see, for example, \cite[Chapter 3.2]{Poo17}). We include a sketch of the proof of the specific statement below.

\begin{lemma}\label{lem: spread out and specialization}
Suppose that $X \to S$ is a morphism between varieties. Let $K = K(S)$ be the field of rational functions on $S$, and let $\bar K$ be the algebraic closure of $K$. If 
\[
\bar g \colon \bar Y \to X_{\bar K}
\] 
is a morphism of varieties over $\bar K$, and $\bar M$ is a coherent sheaf on $\bar Y$, then, after shrinking $S$, there exist a finite \'etale Galois base change $T \to S$, a variety $Y/T$, a morphism 
\[
g \colon Y \to X_T/T,
\] 
and a coherent sheaf $M$ on $Y$ such that $\bar Y \simeq Y_{\bar K}$, $\bar g = g_{\bar K}$, and $\bar M \simeq M_{\bar K}$.
\end{lemma}
\begin{proof}[Sketch of the Proof]
Note that the definitions of the morphism $\bar g$ and the sheaf $\bar M$ involve only finitely many polynomials whose coefficients lie in $\bar K$. Shrinking $S$, there exists a finite \'etale morphism $T \to S$ such that those coefficients become regular functions on $T$. Replacing $T$ by its Galois closure and using the same defining polynomials, we obtain the desired variety $Y/T$, the morphism $g: Y \to X_T/T$, and the coherent sheaf $M$ on $Y$.
\end{proof}

\begin{proposition}\label{prop: Generic property}
Let $f: X \to S$ be a fibration with $X$ a $\Qq$-factorial variety. 
\begin{enumerate}
\item There exist natural maps
\[
\begin{split}
\iota_{\bar K}: &N^1(X/S)_\Rr \to N^1(X_{\bar K})_\Rr, \quad [D] \mapsto [D_{\bar K}],\\
\iota_K: &N^1(X/S)_\Rr \to N^1(X_{K})_\Rr, \quad [D] \mapsto [D_{K}].
\end{split}
\] Moreover, $\iota_K$ is a surjective map.
\item For any sufficiently small open set $U \subset S$, there exists a natural inclusion
\[
N^1(X_U/U)_\Rr \hookrightarrow N^1(X_{\bar K})_\Rr, \quad [D] \mapsto [D_{\bar K}].
\]
\end{enumerate}
\end{proposition}
\begin{proof}
For (1), we show that the natural map
\[
\iota_{\bar K}: N^1(X/S)_\Rr \to N^1(X_{\bar K})_\Rr, \quad [D] \mapsto [D_{\bar K}],
\] is well-defined.

Since $\{D \in \Pic(X/S)_\Rr \mid D \equiv 0/S\}$ is defined over $\Qq$, we only need to show that if a Cartier divisor $D \equiv 0/S$, then $D_{\bar K} \equiv 0$. Replacing $X$ by a resolution $\ti X \to X$ and $D, D_{\bar K}$ by their pullbacks on  $\ti X, \ti X_{\bar K}$ respectively, we can assume that $X$ is smooth.

Let $C_{\bar K} \to X_{\bar K}$ be a smooth curve, we will show $D_{\bar K}\cdot C_{\bar K}=0$. By definition, this is to show that the coefficient of $m$ in the polynomial $\chi(C_{\bar K}, mD_{\bar K})$ is $0$. Shrinking $S$, let $C$ be a spreading out of $C_{\bar K}$ over a variety $T$ such that $h: T \to S$ is a finite \'etale morphism (see Lemma \ref{lem: spread out and specialization}). We can assume that $S$ is smooth and $C$ is smooth over $T$. By Proposition \ref{prop: Generic property 0} (2), we can assume that $X_T$ is normal. Shrinking $T$ further, we may assume that $T = \spec A$ is affine. Moreover, as $\spec \bar K \to T$ is flat, \cite[III Prop 9.3]{Har77} implies that
\[
H^i(C, mD_T) \otimes_A \bar{K} \simeq H^i(C_{\bar K}, mD_{\bar K}).
\] Thus 
\[
\chi(C_{\bar K}, mD_{\bar K}) = \sum (-1)^k \dim_{\bar{K}}H^i(C, mD_T) \otimes_A \bar{K}. 
\] Shrinking $T$, by \cite[III Prop 12.9]{Har77}, we have
\[
H^i(C, mD_{T}) \otimes_A k(t) \simeq H^i({C}_t, mD_t),
\] where $t\in T$ is a closed point. \cite[III Prop 12.9]{Har77} also implies that $H^i(C, mD_T)$ is a free $A$-module. Thus
\[
\begin{split}
\dim H^i(C, mD_T) \otimes_A \bar{K} &= \dim H^i(C, mD_T) \otimes_A k(t) \\
&= \dim H^i({C}_t, mD_t).
\end{split}
\] 

Let $\phi: X_T \to X$ be the natural morphism. Then $D_T \cdot C_t = \phi^*D \cdot C_t = D \cdot \phi_*C_t=0$. Therefore, the coefficient of $m$ in
\[
\chi(C_{\bar K}, mD_{\bar K}) =\chi({C}_t, mD_t)
\] is $0$. This shows that $D_{\bar K} \equiv 0$.

Next, the map $\iota_K$ can be handled by a similar argument. It is surjective because, for any Cartier divisor $D$ on $X_K$, we can take its closure $\bar D$ in $X$. Since $X$ is $\Qq$-factorial, we have $\iota_K(\bar D) = D$.

For (2), in order to get the desired inclusion for any sufficiently small open set, it suffices to find one such open set. We proceed with the argument in several steps.

\noindent Step 1. Suppose that $D_{\bar K}\equiv 0$, we want to find $U$ such that $D\equiv 0/U$ (this $U$ may depend on $D$). Let $\ti X \to X$ be a resolution, and $\ti D$ be the pullback of $D$. We have $\ti D_{\bar K} \equiv 0$ on $\ti X_{\bar K}$. If $\ti D \equiv 0/U$, then $D \equiv 0/U$. Therefore, we can assume that $X$ is smooth. 

By \cite[Theorem 9.6.3 (a) and (b)]{KLe05}, there exists an $m\in\Zz_{>0}$ such that $mD_{\bar K}$ is algebraically equivalent to $0$. That is, there exist connected $\bar{K}$-schemes of finite type $\bar B_i, 1 \leq i \leq n$, invertible sheaves $\bar M_i$ on $X_{\bar B_i}$ and closed points $s_i,t_i$ of $\bar B_i$ such that
\begin{equation}\label{eq: over bar k}
\Oo(mD_{\bar K}) \simeq \bar M_{1,s_1}, ~\bar M_{1,t_1} \simeq \bar M_{2,s_2}, ~\cdots, ~\bar M_{n-1,t_{n-1}} \simeq \bar M_{n,s_n}, ~\bar M_{n,t_n}\simeq \Oo_{X_{\bar K}}
\end{equation} (see \cite[Definition 9.5.9]{KLe05}). Moreover, connecting $s_i, t_i$ by the image of a smooth curve, we can further assume that $\bar B_i$ is a smooth curve.

\noindent Step 2. The desired open set $U$ will be obtained by a sequence of shrinkings of $S$.
\[
\xymatrix{
X_{B_i}\ar[d]_{\theta_i} \ar[r] \ar@/^1.5 pc/[rr]^{\tau_i}& X_T \ar[r]^g \ar[d]& X_U \ar[d] \ar@{^{(}->}[r] & X \ar[d]^f\\
B_i \ar[r]&T\ar@/^1 pc/[l]^{\ti s_i} \ar[r]&U \ar@{^{(}->}[r] & S\\}
\]

First, by generic smoothness, after possibly shrinking $S$ to $U$, we may assume that $X_U \to U$ is smooth. By Lemma \ref{lem: spread out and specialization}, there exists a finite morphism $T \to U$ such that all objects and relations mentioned above over $\bar K$ are defined on $X_T \to T$. This means that there exist schemes $B_i$ over $T$ and invertible sheaves $M_i$ on $X_{B_i}$ such that we have natural isomorphisms $(B_i)_{\bar K} \simeq \bar B_i$ and $(M_i)_{\bar K} \simeq \bar M_i$. Moreover, there exist sections $\ti s_i, \ti t_i: T \to B_i$ which are spreading out of $s_i, t_i$. The spreading out of \eqref{eq: over bar k} becomes 
\[
\begin{split}
\Oo(mD)_{\ti s_1(T)} \simeq M_{1,\ti s_1(T)}, ~&M_{1,\ti t_1(T)} \simeq M_{2,\ti s_2(T)},~ \cdots\\
\cdots, ~&M_{n-1,\ti t_{n-1}(T)} \simeq M_{n,\ti s_n(T)}, ~M_{n,\ti t_n(T)} \simeq \Oo_{\ti t_n(T)},
\end{split}
\] where $\Oo(mD)_{\ti s_i(T)} = \tau_i^*\Oo(mD)|_{\theta_i^{-1}(\ti s_i(T))}$ with $\tau_i: X_{B_i} \to X_T \to X_U$, $\theta_i: X_{B_i} \to B_i$. Note that $X_T \to T$ is isomorphic to both $\theta_i^{-1}(\ti s_i(T)) \to \ti s_i(T)$ and $\theta_i^{-1}(\ti t_i(T)) \to \ti t_i(T)$, where $\theta_i: X_{B_i} \to B_i$. As $\bar B_i$ is smooth over $\bar K$, shrinking $U, T$ further, we can assume that each $B_i$ is also smooth.

After shrinking $T$ (hence also $U$), we will show that $mg^*D \equiv 0/T$, where $g: X_T \to X_U$. Because the intersection is taken in the singular cohomology groups, this can be checked in the analytic topology.  First, as $B_i$ is smooth, shrinking $T$, we can assume that $B_i \to T$ is smooth. As $X_T \to T$ is a smooth morphism, $X_{B_i} \to B_i$ is also a smooth morphism between smooth varieties. Thus $X_{B_i} \to B_i$ is locally trivial in the analytic topology by Ehresmann's theorem. Let $\ell \subset \theta_1^{-1}(\ti s_1(T))$ be a curve which maps to a point on $\ti s_1(T)$. Let $\ell'\subset \theta_1^{-1}(\ti t_1(T))$ be a manifold which is a deformation of $\ell$ in the analytic topology (we do not need $\ell'$ to be an algebraic curve). By induction on $i$, it is enough to show
\[
M_{1,\ti s_1(T)} \cdot \ell = M_{1,\ti t_1(T)} \cdot \ell'.
\] As $M_{1,\ti s_1(T)} \cdot \ell =M_1 \cdot \ell$ and $M_{1,\ti t_1(T)} \cdot \ell'=M_1 \cdot \ell'$, the desired result follows. In particular, we have $D \equiv 0/U$.

\noindent Step 3. To obtain an open set $U$ which is independent of divisors, we can use one of the following two approaches: 

(A) By $\dim N^1(X/S)_\Rr < \infty$, we have
\[
\Ker(N^1(X/S)_\Rr \to N^1(X_{\bar K})_\Rr) < \infty.
\] Let $[D_1], \ldots, [D_e]$ be a basis of $\Ker(N^1(X/S)_\Rr \to N^1(X_{\bar K})_\Rr)$. By the above construction, there exists an open set $U_i$ such that $D_i \equiv 0/U_i$. Then $U \coloneqq \cap_{i=1}^e U_i$ satisfies the desired property.

(B) Replacing $X$ by a resolution, it is enough to show the claim for smooth $X$. Shrinking $U$, we can assume that $X_U \to U$ is smooth. We show that $U$ satisfies the desired property. Let $D$ be any divisor such that $D_{\bar K} \equiv 0$. By the above construction, there exists an open set $V \subset U$ such that $D_V \equiv 0/V$.  We claim that $D \equiv 0/U$. It is enough to show that for any curve $\ell$ such that $\ell$ maps to a point in $U-V$, we have $D \cdot \ell =0$. By Ehresmann's theorem, $\ell$ can be deformed to a complex manifold $\ell'$ in the analytic topology such that $\ell'$ maps to a point in $V$ under $f$. Thus $D\cdot \ell'=D\cdot \ell$. By the dual form of the Lefschetz theorem on $(1,1)$-classes, there exists an algebraic curve $\ell''$ such that $\ell''$ maps to a point in $V$ under $f$ and $D\cdot \ell'=D\cdot \ell''$. Therefore, $D\cdot \ell=D\cdot \ell''=0$.
\end{proof}

We thank Chen Jiang for pointing out that in the following proposition, $W$ is defined over $\Qq$.

\begin{proposition}[{\cite[Proposition 3.8]{Li23}}]\label{prop: max vector defined over Q}
Let $(X, \De) \to S$ be a klt Calabi-Yau fiber space. Let $W$ and $W'$ be the maximal vector spaces in $\bEff(X/S)$ and $\bMov(X/S)$, respectively. Then $W$ and $W'$ are defined over $\Qq$.
\end{proposition}

We can describe $W$ concretely when $R^1f_*\Oo_X=0$.

\begin{proposition}\label{prop: max vector space of Mov}
Let $f: X \to S$ be a fibration with $X$ a $\Qq$-factorial variety. Then $\bAmp(X/S)$ is non-degenerate. If we further assume that $R^1f_*\Oo_X = 0$, then the following results hold.
\begin{enumerate}
\item There is a natural surjective linear map 
\[
r: N^1(X/S)_\Rr \to N^1(X_U/U)_\Rr \quad [D] \mapsto [D|_U].
\] When $U$ is sufficiently small, we have 
\begin{equation}\label{eq: kerr}
\Ker(r) = {\rm Span}_\Rr\{[D] \mid \Supp D \subset \Supp(X-X_U)\}.
\end{equation}
\item The maximal vector space $W \subset \bEff(X/S)$ is generated by divisors in $\Ker(r)$. In particular,  $W \subset \Eff(X/S)$.
\item If $S$ is $\Qq$-factorial, then $\bMov(X/S)$ is non-degenerate.
\end{enumerate}
\end{proposition}
\begin{proof}
Let $D$ be a divisor such that $\pm [D] \in \bAmp(X/S)$. This implies that for any curve $C$ contracted by $f$, we have $\pm D \cdot C \geq 0$. Thus we have $D \equiv 0/S$. Hence, $\bAmp(X/S)$ is non-degenerate.

For (1), note that $r([D])=[D|_U]$ is well-defined. If $D_U$ is a divisor on $X_U$ such that $D_U = \sum c_i B_i$ is the decomposition into irreducible components, then $\overline{D_U} \coloneqq \sum c_i \bar B_i$ is a divisor on $X$ such that $(\overline{D_U})|_U=D_U$. Hence $r$ is surjective.

Let $\bar f: X_{\bar K} \to \spec \bar K$. As $\spec \bar K \to S$ is flat, $R^1\bar f_{*}\Oo_{X_{\bar K}} = (R^1f_*\Oo_X)_{\bar K}=0$. Thus $N^1(X_{\bar K})_\Qq \simeq \Pic(X_{\bar K})_{\Qq}$ (see Lemma \ref{lem: torsion is zero} (3)). As $r$ is defined over $\Qq$, $\Ker(r)$ is also defined over $\Qq$. It is enough to show \eqref{eq: kerr} for Cartier divisors. Take $D$ to be a Cartier divisor such that $[D] \in \Ker(r)$. Shrinking $S$ to $U$ as in Proposition \ref{prop: Generic property}, then by Proposition \ref{prop: Generic property}, we have $D_{\bar K} \equiv 0$. Possibly replacing $D$ by a multiple, we can assume $D_{\bar K} \sim 0$. Thus $D_{\bar K} = \di(\bar\alpha)$ for some $\bar\alpha \in K(X_{\bar K})$. Shrinking $U$ further, by Lemma \ref{lem: spread out and specialization}, there is a finite Galois morphism $T \to U$ such that the above relation is defined on $X_T/T$. In particular, $D_T\coloneqq D|_T = \di(\alpha)$ for some $\alpha \in K(X_T)$. As $D_T$ is $\Gal(X_T/X_U)$-invariant, we have
\[
mD_T = \di(\tau) \text{~with~} \tau \coloneqq \prod_{\theta \in \Gal(X_T/X_U)} \theta(\alpha),
\] where $m = |\Gal(X_T/X_U)|$. As $\tau$ is $\Gal(X_T/X_U)$-invariant, there exists a $\beta \in K(X)$ whose pullback is $\tau$ under the morphism $X_T \to X_U$. Thus $mD_U = \di(\beta)$ on $X_U$. Therefore,
\[
\Supp(mD - \di(\beta)) \subset X - X_U.
\] This shows $``\subset"$ in \eqref{eq: kerr}. The converse inclusion is trivial.

For (2), note that for any $[D] \in \bEff(X/S)$, we have $r([D]) \in \bEff(X_U/U)$. We claim that if $r([D]) \neq 0$, then $[D] \not\in W$. Otherwise, $r([D])\neq 0$ implies that $[D_{\bar K}] \neq 0$ by Proposition \ref{prop: Generic property}. If $[D]\in W$, then $\pm [D_{\bar K}] \in \bEff(X_{\bar K})$. Hence $\bEff(X_{\bar K})$ is degenerate. This is a contradiction as $X_{\bar K}$ is projective. Therefore, $[D] \in W$ implies that $[D] \in \Ker(r)$.

Conversely, let $D$ be an $\Rr$-Cartier divisor such that $$\Supp D \subset \Supp(X-X_U).$$ Then $f(\Supp D) \neq S$. There is an ample divisor $H>0$ on $S$ such that $f(\Supp D) \subset \Supp H$. Thus $D +kf^*H >0$ for some $k \gg 1$ and $[D +kf^*H] = [D] \in \Eff(X/S)$.

For (3), assume that $S$ is $\Qq$-factorial. Let $W' \subset \bMov(X/S)$ be the maximal vector space. We claim that if $0 \neq [D] \in W'$, then there exists a family of curves $\{C_t \mid t\in R\}$ which covers a divisor such that $[C_t] \in N_1(X/S)$ and $D \cdot C_t \neq 0$. By (2), we can assume that $D$ is vertical over $S$. As $S$ is $\Qq$-factorial, there exists an $\Rr$-Cartier divisor $B$ on $S$ such that $D-f^*B$ is a very exceptional divisor. Replacing $D$ by $D-f^*B$, we can assume that $D$  is a very exceptional divisor. Write $D= D^+-D^-$ such that $D^+, D^- \geq 0$ do not have common components. If $D^+ \neq 0$ (resp. $D^- \neq 0$), then by the standard reduce-to-surface argument (for example, see \cite[Lemma 3.3]{Bir12b}), there exists a family of curves $\{C_t \mid t\in R\}$ covering an irreducible component of $\Supp D^+$ (resp. $\Supp D^-$) such that $[C_t] \in N_1(X/S)$ and $D^+ \cdot C_t < 0$ (resp. $D^- \cdot C_t < 0$). Thus $D \cdot C_t < 0$ (resp. $D \cdot C_t > 0$). This shows the claim.

Possibly replacing $D$ by $-D \in W'$, we can assume that $D \cdot C_t<0$. This contradicts with $D \in \bMov(X/S)$.
\end{proof}

\begin{question}
Do the claims in Proposition \ref{prop: max vector space of Mov} still hold true for an arbitrary fibration $f: X \to S$?
\end{question}

We will need the following lemma in the sequel.

\begin{lemma}\label{lem: torsion is zero}
Let $f: X \to S$ be a fibration.
\begin{enumerate}
\item (\cite[Corollary 7.8]{Kol86}) If $X$ and $S$ have rational singularities, then $R^1f_*\Oo_X$ is torsion free.
\item If $(X, \De) \to S$ is a klt Calabi-Yau fibratioin, then $X$ and $S$ both have rational singularities.
\item If $H^1(X_{\bar K}, \Oo_{X_{\bar K}})=0$, then $\Pic({X_{\bar K}})_\Qq = N^1({X_{\bar K}})_\Qq$.
\end{enumerate}
\end{lemma}
\begin{proof}
(1) Let $\tau: S' \to S$ and $\sigma: X' \to X$ be resolutions such that there exists a fibration $f': X' \to S'$ with $\tau \circ f' = f \circ \sigma$. By \cite[Corollary 7.8]{Kol86}, $R^1f'_*\Oo_{X'}$ is a torsion free sheaf. As $X, S$ have rational singularities, Leray spectral sequence implies that
\[
R^1(f\circ \sigma)_*\Oo_{X'}=R^1f_*\Oo_X \quad \text{~and~} \quad R^1(\tau \circ f')_*\Oo_{X'}= \tau_*R^1f_*'\Oo_{X'}. 
\] As $\tau_*R^1f_*'\Oo_{X'}$ is torsion free, $R^1f_*\Oo_X$ is torsion free.

(2) By \cite[Theorem 0.2]{Amb05}, there exists a klt pair $(S, B)$ such that $K_X+\De \sim_\Rr f^*(K_S+B)/S$. As klt singularities are rational (see \cite[Theorem 5.22]{KM98}), $X$ and $S$ both have rational singularities.

(3) Let $\Pic^0(X_{\bar K})$ be the identity component of the Picard scheme of $X_{\bar K}$. We have $\dim \Pic^0(X_{\bar K})=\dim H^1(X_{\bar K}, \Oo_{X_{\bar K}})$. Hence $H^1(X_{\bar K}, \Oo_{X_{\bar K}})=0$ implies that $\Pic({X_{\bar K}})_\Qq = N^1({X_{\bar K}})_\Qq$.
\end{proof}

Recall that $\Gamma_B$ is the image of $\PsAut(X/S, \De)$ under the natural group homomorphism $\iota: \PsAut(X/S, \De) \to {\rm GL}(N^1(X/S)_\Rr)$.

\begin{lemma}\label{le: trivial kernel}
Let $f: X \to S$ be a Calabi-Yau fiber space such that $X$ has terminal singularities. Assume that $R^1f_*\Oo_X=0$. Let $W\subset \bMov(X/S)$ be the maximal vector space. Then
\[
\Gamma_W\coloneqq \{\gamma \in \Gamma_B \mid  \gamma \text{~acts trivially on ~}N^1(X/S)_\Rr/W\}
\] is a finite group.
\end{lemma}

\begin{proof}
As $R^1f_*\Oo_X=0$, we have $H^1(X_{\bar K},\Oo_{X_{\bar K}})=0$ and thus $\Pic(\bar X)_\Qq \simeq N^1(X_{\bar K})_\Qq$ by Lemma \ref{lem: torsion is zero} (3). Let $G\coloneqq \{g\in \PsAut(X/S) \mid \iota(g)\in \Gamma_W\}$. It suffices to show that $G$ is a finite set. By Proposition \ref{prop: max vector space of Mov} (1) and (2), there exists an open set $U \subset S$ such that $N^1(X/S)_\Rr/W \to N^1(X_U/U)_\Rr$ is surjective. Let $H$ be an ample$/S$ divisor on $X_U$. Then $g\cdot H \equiv H/U$ for any  $g\in G$. Thus $g_{\bar K} \cdot H_{\bar K} \equiv H_{\bar K}$ in $N^1(X_{\bar K})$, where $g_{\bar K}$ and $H_{\bar K}$ correspond to $g$ and $H$ respectively after the base change. 

We claim that $\{g_{\bar K} \mid g\in G\}$ is a finite set. First, as $g_{\bar K}$ is isomorphic in codimension $1$ and $g_{\bar K} \cdot H_{\bar K} \equiv H_{\bar K}$, we see that $g_{\bar K} \in \Aut(X_{\bar K})$. Let 
\[
\Aut_{H_{\bar K}}(X_{\bar K}) \coloneqq \{h\in \Aut(X_{\bar K}) \mid h\cdot H_{\bar K} \equiv H_{\bar K} \}
\] be a sub-scheme of the group scheme $\Aut(X_{\bar K})$, then $\Aut_{H_{\bar K}}(X_{\bar K})$ is a scheme of finite type over $\bar K$ (see, for example, \cite[Remark 2.6]{MZ18}). Let $\Aut^0(X_{\bar K})$ be the identity component of the group scheme $\Aut(X_{\bar K})$, then \cite[Theorem 4.5]{Xu20} shows that $\dim \Aut^0(X_{\bar K}) =\dim H^1(X_{\bar K}, \Oo_{X_{\bar K}})=0$. Hence, $\Aut(X_{\bar K})$ is a discrete group and thus $\Aut_{H_{\bar K}}(X_{\bar K})$ is a finite group. This implies that $\{g_{\bar K} \mid g\in G\}$ is a finite set.

Finally, for $g, h\in G$, if $g_{\bar K} = h_{\bar K}$, then $g=h$. Thus $G$ is also a finite set. 
\end{proof}

\begin{remark}The group $\Gamma_W$ may not be trivial. Consider a Calabi-Yau threefold $X$ which is a general member in the linear system of $|\Oo_{\Pp^2 \times \Pp^1 \times \Pp^1}(3,2,2)|$. \cite[Example 3.8 (4)]{Kaw97} shows that the natural projection $X \to \Pp^2$ is an elliptic fiberation which admits a sequence of flops 
\[
\gamma_1 \circ \gamma_2 \circ \gamma_1 \cdots: X \dto X \dto \cdots \dto X
\] over $\Pp^2$, where $\gamma_1, \gamma_2 \in \Bir(X/\Pp^2)$. In particular, for each $X \to \Pp^2$, we have $N^1(X/\Pp^2)_\Rr=\Mov(X/\Pp^2)=W=\Rr$. Thus $\Gamma_W= \{\pm 1\}$ which acts trivially on $N^1(X/\Pp^2)_\Rr/W$.
\end{remark}

\section{A variant of the cone conjecture}\label{sec: A variant of the cone conjecture}

In this section, we study the relationship between the cone conjecture and Conjecture \ref{conj: shokurov polytope}. Note that in Conjecture \ref{conj: shokurov polytope}, by enlarging $P_M$ and $P_A$, we can always assume that $P_M$ and $P_A$ are rational polyhedral cones. Recall that a polyhedral cone is closed by definition and $\Gamma_B$ (resp. $\Gamma_A$) is the image of $\PsAut(X/S, \De)$ (resp. $\Aut(X/S,\De)$) under the group homomorphism $\PsAut(X/S, \De) \to {\rm GL}(N^1(X/S)_\Rr)$. By Definition \ref{def: polyhedral type}, we set
\[
\begin{split}
&\Mov(X/S)_+ \coloneqq \Conv(\bMov(X/S) \cap N^1(X/S)_\Qq),\\
&\Amp(X/S)_+ \coloneqq \Conv(\bAmp(X/S) \cap N^1(X/S)_\Qq).
\end{split}
\] 

We thank the referee for providing the proof of Lemma \ref{lem: inclusion} (2), which weakens the original assumption.

\begin{lemma}\label{lem: inclusion}
Let $f: (X, \De) \to S$ be a klt Calabi-Yau fiber space. 
\begin{enumerate}
\item (\cite[Theorem 2.15]{LOP20}) We have $\bAmp^e(X/S) \subset \Amp(X/S)_+$.
\item Assume the existence of minimal models for effective klt pairs in $\dim(X/S)$, then $\bMov^e(X/S) \subset \Mov(X/S)_+$
\end{enumerate}
\end{lemma}
\begin{proof}
For $[D] \in \Eff(X/S)$, replacing $D$ by a divisor which is numerically equivalent to $D$, we can assume that the irreducible decomposition of $D =\sum_{i=1}^k a_i D_i$ with $a_i>0$. Let $P \coloneqq  \Cone([D_i] \mid i=1, \ldots, k) \subset \Eff(X/S)$ be a rational polyhedral cone.

The (1) is shown in \cite[Theorem 2.15]{LOP20}. We include an argument for the reader’s
convenience. Assume that $[D] \in \bAmp^e(X/S)$. By Theorem \ref{thm: nef cone is polyhedral}, $P_N = P \cap \bAmp(X/S)$ is a rational polyhedral cone. Thus $$[D] \in P_N \subset \Amp(X/S)_+.$$

For (2), by the assumption on the existence of minimal models for effective klt pairs, we obtain
\[
\bMov^e(X/S) \;=\; \bigcup_{\alpha \colon X \dto X' \,\text{small $\Qq$-factorial modification$/S$}} \alpha^* \bAmp^e(X'/S).
\] By (1), we have $\bAmp^e(X'/S) \subset \Amp(X'/S)_+$. As $\alpha^*: N^1(X'/S)_\Qq \to N^1(X/S)_\Qq$ is an isomorphism of $\Qq$-vector spaces, we have $\alpha^*\bAmp^e(X'/S) \subset \alpha^*\Amp(X'/S)_+ \subset \Mov(X'/S)_+$.
\end{proof}

\begin{lemma}\label{le: shrink to fundamental domain}
Let $f: (X, \De) \to S$ be a klt Calabi-Yau fiber space. 
\begin{enumerate}
\item Assume the existence of good minimal models for effective klt pairs in $\dim(X/S)$.  If there exists a rational polyhedral cone $P_M \subset \Eff(X/S)$ satisfying Conjecture \ref{conj: shokurov polytope} (1), then there is a rational polyhedral cone $Q_M \subset \Mov(X/S) \cap P_M$ such that
\begin{equation}\label{eq: equal after action}
\bigcup_{g \in \PsAut(X/S, \De)} g\cdot Q_M =  \Mov(X/S).
\end{equation}
\item If there exists a rational polyhedral cone $P_A \subset \Eff(X/S)$ satisfying Conjecture \ref{conj: shokurov polytope} (2), then there is a rational polyhedral cone $Q_A \subset \bAmp^e(X/S) \cap P_A$ such that
\begin{equation}\label{eq: equal after action 2}
\bigcup_{g \in \Aut(X/S,\De)} g\cdot Q_A =  \bAmp^e(X/S).
\end{equation}
\end{enumerate}
\end{lemma}

\begin{proof}
For (1), by Theorem \ref{thm: Shokurov-Choi}, $P_M=\cup_{i=0}^mP^\circ_i$ is a union of finitely many relatively open rational polyhedral cones. Let $P^\circ_1, \ldots, P^\circ_k$ be the polyhedral cones such that $P^\circ_j \cap \bMov(X/S) \neq\emptyset$.

 We claim that $P_j \coloneqq \overline{P_j^\circ}\subset \Mov(X/S)$. Let $D \geq 0$ such that $[D] \in P^\circ_j \cap \bMov(X/S)$. Assume that $(Y/S, \De_Y+\ep D_Y)$ is a weak log canonical model of $(X/S, \De+\ep D)$ for some $\ep \in \Qq_{>0}$. By Lemma \ref{le: lift to iso in codim 1}, we can assume that $X, Y$ are isomorphic in codimension $1$. Take $[B'] \in P_j$, by $P_M \subset \Eff(X/S)$, there exists a sequence $\{[B_l]\}_{l \in \Nn} \subset P_j^\circ$ with $B_l \geq 0$ such that $\lim_{l \to +\infty} [B_l] = [B']$ and $\lim_{l \to +\infty} B_l = B$ as the limit of Weil divisors. Thus $[B']=[B]$. By Theorem \ref{thm: Shokurov-Choi}, there exists a $\delta \in \Qq_{>0}$ such that $(Y/S, \De_Y+\delta B_{l,Y})$ is a weak log canonical model of $(X/S, \De+\delta B_l)$ for each $l$. Thus $(Y/S, \De_Y+ \delta B_Y)$ is also a weak log canonical model of $(X/S, \De+\delta B)$. Indeed, $K_Y+\De_Y+\delta B=\lim_{l \to +\infty} K_Y+\De_Y+\delta B_{l,Y}$ is nef over $S$, and for any prime divisor $D$ over $Y$, the log discrepancies satisfy 
 \[
 a(D; X, \De+\delta B) = \lim_{l \to +\infty} a(D; X, \De+\delta B_l) \leq \lim_{l \to +\infty} a(D; Y, \De_Y+\delta B_{l,Y})=  a(D; Y, \De_Y+\delta B_Y). 
 \] By Theorem \ref{thm: HX13}, $B_Y$ is semi-ample$/S$. Hence $B$ is movable, and thus $[B] \in \Mov(X/S)$.
 
Let $Q_M \coloneqq \Cone(P_1, \ldots, P_k)$ be the cone generated by $P_j, 1\leq  j \leq k$. Then $Q_M \subset \Mov(X/S)$. For $[M] \in \Mov(X/S)$, there exists a $g\in \PsAut(X/S, \De)$ such that $g\cdot [M] \in P_M$. Thus $g\cdot [M] \in P_j^\circ$ for some $j$ and hence $g\cdot [M] \in Q_M$. This shows \eqref{eq: equal after action}.

For (2), Theorem \ref{thm: nef cone is polyhedral} shows that $Q_A \coloneqq P_A \cap \bAmp(X/S)$ is a rational polyhedral cone. By $P_A \subset \Eff(X/S)$, we have $Q_A \subset \bAmp^e(X/S)$. For any $[H] \in \Amp(X/S)$, there exist an $[H'] \in P_A$ and a $g\in \Aut(X/S,\De)$ such that $g\cdot [H']= [H]$. Hence $[H'] \in Q_A$. Thus $\Gamma_A\cdot Q_A \supset \Amp(X/S)$. As $\bAmp(X/S)$ is non-degenerate (see Proposition \ref{prop: max vector space of Mov}) and $Q_A \subset \Amp(X/S)_+$ by Lemma \ref{lem: inclusion} (1), Proposition \ref{prop: prop-def} implies that $\Gamma_A\cdot Q_A = \Amp(X/S)_+ \supset \bAmp^e(X/S)$. The $``\subset"$ of \eqref{eq: equal after action 2} follows from the definition. 
\end{proof}

\begin{proposition}\label{prop: fun domain of Mov}
Let $f: (X,\De) \to S$ be a klt Calabi-Yau fiber space. Let $W\subset \bMov(X/S)$ be the maximal vector space. Assume that good minimal models exist for effective klt pairs in dimension $\dim(X/S)$. Suppose that there is a polyhedral cone $P \subset \Mov(X/S)$ such that $$\PsAut(X/S, \De) \cdot P = \Mov(X/S).$$ We have the following results. 
\begin{enumerate}
\item If either $R^1f_*\Oo_X=0$ or $W=0$, then we have $$\Mov(X/S) = \bMov^e(X/S)=\Mov(X/S)_+.$$
\item There are finitely many varieties $Y_j/S, j\in J$ such that if $X \dto Y/S$ is isomorphic in codimension $1$ with $Y$ a $\Qq$-factorial variety, then $Y \simeq Y_j/S$ for some $j \in J$.
\item If $\bMov(X/S)$ is non-degenerate, then $\bMov^e(X/S)$ has a rational polyhedral fundamental domain under the action of $\Gamma_B$.
\item If $R^1f_*\Oo_X=0$, then $\bMov^e(X/S)$ has a weak rational polyhedral fundamental domain (maybe degenerate) under the action of $\Gamma_B$. 
\end{enumerate}
\end{proposition}
\begin{proof}
Possibly enlarging $P$, we can assume that $P$ is a rational polyhedral cone.

For (1), we have $``\subset"$ for the above three cones by Lemma \ref{lem: inclusion}. By Proposition \ref{prop: max vector defined over Q}, $W$ is defined over $\Qq$. By definition, $\Mov(X/S) \supset  \Int(\bMov(X/S))$. Then  $\Gamma_B \cdot P \supset  \Int(\bMov(X/S))$. Thus $(\Mov(X/S)_+, \Gamma_B)$ is of polyhedral type. We follow the notation of Lemma \ref{le: induced polyhedral type}. By Lemma \ref{le: induced polyhedral type} (3) and Proposition \ref{prop: prop-def}, we have \begin{equation}\label{eq: quotient equal}
\Gamma_B \cdot \ti P = (\widetilde{\Mov(X/S)})_+ = {(\Mov(X/S)_+)\ \widetilde{}} \ ,
\end{equation} where the last equality follows from Lemma \ref{le: induced polyhedral type} (2). 

We claim that $W \subset\Mov(X/S)$. By Proposition \ref{prop: max vector space of Mov} (1) and (2), $W \subset \Eff(X/S)$. Let $[D] \in W$ be a rational point such that $D \geq 0$. Then for a sufficiently small $\ep\in \Qq_{>0}$, $(X/S, \De+\ep D)$ has a weak log canonical model $(Y/S, \De+\ep D_Y)$. Because $[D] \in \bMov(X/S)$, by Lemma \ref{le: lift to iso in codim 1}, we can assume that $X, Y$ are isomorphic in codimension $1$. Note that $D_Y$ is semi-ample$/S$ by Theorem \ref{thm: HX13}. Thus $[D] \in \Mov(X/S)$. As $W$ is $\Gamma_B$-invariant and $\Gamma_B\cdot (P+W) = \Mov(X/S)_+$ by \eqref{eq: quotient equal}, we have $\Mov(X/S) = \Mov(X/S)_+$.

For (2), there exists a decomposition $P= \cup_{i=1}^k P_k^\circ$ as in Theorem \ref{thm: Shokurov-Choi}. For each $j$, by Lemma \ref{le: lift to iso in codim 1} and Theorem \ref{thm: HX13}, we can choose a $f_j: X \dto Y_j/S$ which is isomorphic in codimension $1$ such that if $[D]\in P_j^\circ$ with $D \geq 0$, then $(Y_j/S, \De_{Y_j}+\ep D_{Y_j})$ is a $\Qq$-factorial weak log canonical model of $(X/S, \De+\ep D)$ for some $\ep \in \Qq_{>0}$. We claim that if $g: Y \dto X/S$ is isomorphic in codimension $1$, then $Y \simeq Y_j/S$ for some $j$. In fact, let $A \geq 0$ be an ample$/S$ divisor on $Y$. Then $g_*A \in \Mov(X/S)$. Let $\sigma \in \PsAut(X/S, \De)$ such that $\sigma \cdot g_*A \in P$. Then $\sigma \cdot g_*A \in P_j^\circ$ for some $j$. Note that $Y, Y_j$ are $\Qq$-factorial varieties. Because $(\sigma \cdot g_*A)_{Y_j}=(f_j \circ \sigma \circ g)_*A$ is nef/$S$ and 
\[
f_j \circ \sigma \circ g: Y \dto X \dto X \dto Y_j/S
\] is isomorphic in codimension $1$, we have $Y \simeq Y_j/S$.

For (3) and (4), note that $(\Mov(X/S)_+, \Gamma_B)$ is of polyhedral type. By Proposition \ref{prop: max vector defined over Q} and Proposition \ref{prop: degenerate cone}, there is a rational polyhedral cone $\Pi$ such that $\Gamma_B \cdot \Pi = \Mov(X/S)_+$, and for each $\gamma \in \Gamma_B$, either $\gamma \Pi\cap \Int(\Pi) = \emptyset$ or $\gamma \Pi = \Pi$. Moreover, 
\[
\{\gamma \in \Gamma_B \mid \gamma \Pi = \Pi\} = \{\gamma \in \Gamma_B \mid \gamma \text{~acts trivially on ~}N^1(X/S)_\Rr/W\}.
\] Hence $\Pi$ is a weak rational polyhedral fundamental domain. In particular, if $W=0$, then $\Pi$ is a rational polyhedral fundamental domain.
\end{proof}

\begin{remark}
The assumption in Proposition \ref{prop: fun domain of Mov} (1) is necessary. \cite[Example 3.8 (2)]{Kaw97} gives an elliptic fibration (hence $R^1f_*\Oo_X \neq 0$) with $W \neq 0$ such that $\Mov(X/S) = \bMov^e(X/S) \neq \Mov(X/S)_+$. In this example, $W$ is defined over $\Qq$ but $W \not\subset \Mov(X/S)$.
\end{remark}

\begin{proposition}\label{prop: fun for Amp}
Let $f: (X,\De) \to S$ be a klt Calabi-Yau fiber space. Suppose that there is a polyhedral cone $P \subset \bAmp^e(X/S)$ such that $\Aut(X/S,\De) \cdot P = \bAmp^e(X/S)$. We have the following results. 
\begin{enumerate}
\item There are finitely many varieties $Y_j/S, j \in J$ such that if $X \to Z/S$ is a surjective fibration to a normal variety $Z$, then $Y_j \simeq Z/S$ for some $j \in J$.
\item The cone $\bAmp^e(X/S)$ has a rational polyhedral fundamental domain.
\end{enumerate}
\end{proposition}

This result can be shown analogously as Proposition \ref{prop: fun domain of Mov}  and thus we only sketch the proof.

\begin{proof}[Sketch of the Proof]
For (1), let $A$ be an ample$/S$ divisor on $Z$. Then for a morphism $g: X \to Z/S$, $g^*A$ lies in $\bAmp^e(X/S)$. There exists $\theta \in\Aut(X/S,\De)$ such that $[\theta \cdot g^*A]$ lies in the interior of a face $F\subset P$. The morphism $g\circ \theta^{-1}:X \to Z$ corresponds to the contraction of $F$. As $P$ is a polyhedral cone, there are only finitely many faces.

(2) follows from Lemma \ref{le: existence of fun domain} as $\Amp(X/S)$ is non-degenerate by Proposition \ref{prop: max vector space of Mov}.
\end{proof}

We have the following remark regarding the cones chosen in the statement of the cone conjecture (cf. \cite[Section 3]{LOP18}):

\begin{remark}\label{rmk: choice of cones}
Let $f: (X,\De) \to S$ be a klt Calabi-Yau fiber space. Assuming that good minimal models of effective klt pairs exist in dimension $\dim(X/S)$ and either $R^1f_*\Oo_X=0$ or $\bMov(X/S)$ is non-degenerate, 
Lemma \ref{le: shrink to fundamental domain} and Proposition \ref{prop: fun domain of Mov} (1) imply that $\Mov(X/S)$ has a (weak) rational polyhedral fundamental domain iff $\bMov^e(X/S)$ has a (weak) rational polyhedral fundamental domain. 

Therefore, at least when $S$ is a point, modulo the standard conjectures in the minimal model program, there is no difference to state the cone conjecture of movable cones for either $\Mov(X/S)$ or $\bMov^e(X/S)$. 

If $\Amp(X/S)$ has a (weak) rational polyhedral fundamental domain, then $(\Amp(X/S)_+, \Gamma_A)$ is of polyhedral type. Proposition \ref{prop: prop-def} and Lemma \ref{lem: inclusion} imply that $$\Amp(X/S)=\bAmp^e(X/S)=\Amp(X/S)_+.$$ Therefore, $\bAmp^e(X/S)$ has a rational polyhedral fundamental domain by Lemma \ref{le: existence of fun domain}. In summary, the cone conjecture for $\Amp(X/S)$ implies that for $\bAmp^e(X/S)$.

However, a priori, $\Mov(X/S)_+$ (resp. $\Amp(X/S)_+$) has a rational polyhedral fundamental domain $\Pi$ may not imply that $\bMov^e(X/S)$ (resp. $\bAmp^e(X/S)$) has a rational polyhedral fundamental domain. More importantly, under this assumption, we only know $\Pi \subset \bEff(X/S)$, hence Theorem \ref{thm: Shokurov-Choi} and Theorem \ref{thm: nef cone is polyhedral} do not apply in this setting. Therefore, the argument of finiteness of birational models which are isomorphic in codimension $1$ (resp. finiteness of contraction morphisms) breaks. It is for this reason that we do not state the cone conjectures for $\Mov(X/S)_+$ and $\Amp(X/S)_+$.
\end{remark}

The above discussions lead to the proof of Theorem \ref{thm: main 1}. 

\begin{proof}[Proof of Theorem \ref{thm: main 1}]
The (1) and (2) follow from Lemma \ref{le: shrink to fundamental domain} and Proposition \ref{prop: fun domain of Mov} (4) and (3). The (3)  follows from Lemma \ref{le: shrink to fundamental domain} and Proposition \ref{prop: fun for Amp} (2).
\end{proof}

\section{Generic and Geometric cone conjectures}\label{sec: Generic and Geometric cone conjecture}

\subsection{Generic cone conjecture}\label{subsec: generic cone conj}

For a Calabi-Yau fiber space, we study the relationship between the relative cone conjecture and the cone conjecture of its generic fiber. Conjecture \ref{conj: shokurov polytope} is especially convenient to study movable cones in the relative setting. Hence we only focus on the cone conjecture for movable cones in this section. 

Let $f: X \to S$ be a Calabi-Yau fiber space. Recall that $K\coloneqq K(S)$ is the field of rational functions of $S$, and $X_K \coloneqq X \times_{S} \spec K$.

\begin{theorem}\label{thm: generic cone conjecture}
Let $f: X \to S$ be a Calabi-Yau fiber space such that $X$ has terminal singularities. Suppose that good minimal models of effective klt pairs exist in dimension $\dim(X/S)$. Assume that $R^1f_*\Oo_X=0$. 

If the weak cone conjecture holds true for $\bMov^e(X_K/K)$, then the weak cone conjecture holds true for $\bMov^e(X/S)$. 

Moreover, if $\Mov(X/S)$ is non-degenerate, then the cone conjecture holds true for $\bMov^e(X/S)$. In particular, if $S$ is $\Qq$-factorial, then the cone conjecture holds true for $\bMov^e(X/S)$. 
\end{theorem}
\begin{proof}
Let $\Pi_K \subset \bMov^e(X_K/K)$ be a polyhedral cone such that $$\PsAut(X_K/K) \cdot \Pi_K= \Mov(X_K/K).$$ Let $\Pi \subset \Eff(X/S)$ be a polyhedral cone which is a lift of $\Pi_K$. In other words, $\Pi$ maps to $\Pi_K$ under the natural surjective map $N^1(X/S)_\Rr \to N^1(X_K/K)_\Rr$ (see Proposition \ref{prop: Generic property} (1)). 

If $g_K \in \PsAut(X_K/K)$, then $g_K$ can be viewed as a birational morphism $g$ of $X$ over $S$. Then $g \in \Bir(X/S)=\PsAut(X/S)$ as $K_X$ is nef$/S$ and $X$ has terminal singularities. Indeed, let $p: W \to X, q: W \to X$ be resolutions such that $g=q \circ p^{-1}$. We have $K_W=p^*K_X+E$ and $K_W=q^*K_X+F$ with $\Supp E =\Exc(p), \Supp F= \Exc(q)$ as $X$ has terminal singularities. As $K_X$ is nef$/S$, we have $E=F$ by the negativity lemma. Thus $g$ induces an isomorphism $X\backslash p(E) \simeq X\backslash q(F)$. This shows $g\in \PsAut(X/S)$.

Let $W \subset \bMov(X/S)$ be the maximal vector space. We claim that for $P \coloneqq \Cone(\Pi \cup W)$, 
\begin{equation}\label{eq: desired inclusion}
P \subset \Eff(X/S)  \text{~and~} \PsAut(X/S) \cdot P \supset \Mov(X/S).
\end{equation}
By Proposition \ref{prop: max vector space of Mov} (2), $W$ is generated by vertical divisors and thus $W \subset \Eff(X/S)$. This shows $P \subset \Eff(X/S)$. Next, for any $[M] \in \Mov(X/S)$ such that $M$ is an $\Rr$-Cartier divisor, there exist an $\Rr$-Cartier divisor $D$ on $X$ and a $g\in \PsAut(X/S)$ such that $[D] \in \Pi$ and $g_K \cdot [D_K]=[M_K]$. As $R^1f_*\Oo_X=0$, we have $g_K \cdot D_K \sim_\Rr M_K$ by Lemma \ref{lem: torsion is zero}. By Proposition \ref{prop: Generic property} (1), there exists a vertical divisor $B$ on $X$ such that $g\cdot D +B \sim_\Rr M/S$. Thus $D+g^{-1}\cdot B \in P$ and $g\cdot [D+g^{-1} \cdot B]=[M]$. This shows $\PsAut(X/S) \cdot P \supset \Mov(X/S)$.

The \eqref{eq: desired inclusion} shows that Conjecture \ref{conj: shokurov polytope} (1) is satisfied. Then Theorem \ref{thm: main 1} (1) and (2) imply the desired claim. Note that by Proposition \ref{prop: max vector space of Mov} (3), if $S$ is $\Qq$-factorial, then $\bMov(X/S)$ is non-degenerate. 
\end{proof}

\begin{remark}
The above argument does not work for a log pair $(X/S, \De)$ because each $g\in \PsAut(X_K/K, \De_K)$ may not lift to $\PsAut(X/S,\De)$.
\end{remark}

Now Theorem \ref{thm: K3} follows from Theorem \ref{thm: generic cone conjecture} and the cone conjecture of K3 surfaces over arbitrary fields with characteristic $\neq 2$ (\cite{BLL20}). 

\begin{proof}[Proof of Theorem \ref{thm: K3}]
We have $R^1f_*\Oo_X \otimes k(t) \simeq H^1(X_t, \Oo_{X_t})=0$, where $t\in S$ is a general closed point. Hence $R^1f_*\Oo_X$ is a torsion sheaf and thus $R^1f_*\Oo_X=0$ by Lemma \ref{lem: torsion is zero} (1).

We claim that $X_{K}$ is a smooth K3 surface. Let $U \subset S$ be a smooth open set such that $X_U \to U$ is flat and for any closed point $t \in U$, $X_t$ is a K3 surface. By 
\cite[\href{https://stacks.math.columbia.edu/tag/01V8}{Lemma 01V8}]{Sta22}, $f_U: X_U \to U$ is a smooth morphism. Thus $X_K/K$ is smooth. Note that
\[
\spec {K(S)} \to U, \quad \spec {k(t)}  \to U
\] are flat morphisms, where $t\in U$ is a closed point. Then \cite[III Prop 9.3]{Har77} implies that for a quasi-coherent sheaf $\Ff$ on $X_U$ and $i \geq 0$,
\begin{equation}\label{eq: base change}
\begin{split}
H^i(X_U, \Ff) &\otimes_U {K} \simeq H^i(X_{K}, \Ff_{X_{K}}),\\
H^i(X_U, \Ff) &\otimes_U {k(t)} \simeq H^i(X_t, \Ff_{X_t}).
\end{split}
\end{equation} First, applying \eqref{eq: base change} to $\omega_{X_U/U}, \omega_{X_U/U}^{-1}$ and $i=0$, we have $\Oo_{X_K}(K_{X_K}) \simeq \Oo_{X_K}$. Next, applying \eqref{eq: base change} to $\Oo_{X_U}$ and $i=1$, we have $H^1(X_K, \Oo_{X_K})=0$. This shows that $X_K/K$ is a K3 surface. 

We claim that $\Amp(X_K/K)_+=\bAmp^e(X_K/K)$. By Lemma \ref{lem: inclusion} (1), it suffices to show that $\Amp(X_K/K)_+ \subset \Eff(X_K/K)$. Let $D_K$ be a Cartier divisor on $X_K$ such that $[D_K] \in \bAmp(X_K/K)\cap N^1(X_K/K)_\Qq$. A similar argument as above shows that $X_{\bar K}$ is a K3 surface over $\bar K$. An application of Riemann-Roch shows that there exists an effective divisor $E_{\bar K}$ such that $D_{\bar K} \sim E_{\bar K}$. As cohomology is invariant under flat base change, we have
\[
h^0(X_K, \Oo_{X_K}(mD_K)) \;=\; h^0(X_{\bar K}, \Oo_{X_{\bar K}}(mD_{\bar K})) \text{~for all~} m\in \Zz.
\] This implies see $[D_K]\in \Eff(X_K/K)$.

By \cite[Corollary 3.15]{BLL20}, there is a rational polyhedral cone $\Pi \subset \Amp(X_K/K)_+$ which is a fundamental domain of $\Amp(X_K/K)_+$ under the action of $\Aut(X_K/K)$.
By $\Amp(X_K/K)_+=\bAmp^e(X_K/K)=\bMov^e(X_K/K)$ and $\Aut(X_K/K)=\PsAut(X_K/K)$, Theorem \ref{thm: generic cone conjecture} implies the desired result. 
\end{proof}

\begin{remark}\label{rmk: hyperkahler}
For a projective hyperk\"ahler manifold $X$ over a characteristic zero field $k$, \cite[Theorem 1.0.5]{Tak21} showed that $\Mov(X/k)_+$ has a rational polyhedral fundamental domain $\Pi$ under the action of $\Bir(X/k)$. However, this is not sufficient to deduce the cone conjecture for movable cones of hyperk\"ahler fibrations. Indeed, we do not know that $\Pi \subset \Eff(X/k)$. On the other hand, when $k=\Cc$, it is conjectured that $\Mov(X/k)_+\subset \Eff(X/k)$ (see \cite[Conjecture 1.4]{BM14}). We refer to \cite{HPX24} for the latest developments on the cone conjecture for irreducible holomorphic symplectic manifolds.
\end{remark}

\subsection{Geometric cone conjecture}\label{subsec: geometric cone conj}
For a klt Calabi-Yau fiber space, we study the relationship between the relative cone conjecture and the cone conjecture of its geometric fiber. We show that the cone conjecture for the movable cone of the geometric fiber implies the relative cone conjecture for the movable cone of a finite Galois base change. 

When $X$ is not a $\Qq$-factorial variety, if $X$ admits a small $\Qq$-factorization $\ti X \to X$, then the (weak) cone conjecture for $\bMov^e(X/S)$ is understood as the (weak) cone conjecture for $\bMov^e(\ti X/S)$. By \cite[Corollary 1.4.3]{BCHM10}, such small $\Qq$-factorization exists if there is a divisor $\De$ such that $(X,\De)$ is klt. Although the cone conjecture for $\bMov^e(X/S)$ still makes sense for non-$\Qq$-factorial varieties, the $\Qq$-factoriality provides more convenience while preserving the geometric consequences of the cone conjecture (e.g. the finiteness of birational contraction models). Note that the movable cones and pseudo-automorphism groups of different small $\Qq$-factorizations are naturally identified. Hence, the validity of the conjecture is independent of the choice of $\ti X$. On the other hand, the ample cones and automorphism groups of non-isomorphic small $\Qq$-factorizations cannot be identified. Therefore, we do not pass to a small $\Qq$-factorization when considering the cone conjecture for $\bAmp^e(X/S)$.

For a klt Calabi-Yau fiber space $f: (X,\De) \to S$. Let $K=K(S)$ and $\bar K$ be the algebraic closure of $K$. For $g\in \Bir(X/S)$, let $g_{\bar K} \in \Bir(X_{\bar K}/\bar K)$ be the extension of $g$ under the base change $\spec \bar K \to S$. Let $X_{\bar K}\coloneqq X \times_S \spec \bar K$ be the geometric fiber of $f$. Set $\De_{\bar K} \coloneqq \De \times_S  \spec \bar K$. By Proposition \ref{prop: Generic property 0} (1), $(X_{\bar K},\De_{\bar K})$ is still klt. Note that even if $X$ is $\Qq$-factorial, $X_{\bar K}$ may not be $\Qq$-factorial. Let $\ti\pi:\widetilde{X_{\bar K}} \to X_{\bar K}$ be a small $\Qq$-factorization. Set $\widetilde{\De_{\bar K}}$ be the strict transform of $\De_{\bar K}$. Let $\bar\Gamma_B$ be the image of $\PsAut(\widetilde{X_{\bar K}}/\bar K, \widetilde{\De_{\bar K}})(=\PsAut(X_{\bar K}/\bar K, \De_{\bar K}))$ under the group homomorphism
\begin{equation}\label{eq: tau}
\iota_{\bar K}: \PsAut(\widetilde{X_{\bar K}}/\bar K, \widetilde{\De_{\bar K}}) \to {\rm GL}(N^1(\widetilde{X_{\bar K}}/\bar K)_\Rr).
\end{equation} 

\begin{proposition}\label{prop: realize actions}
Under the above notation and assumptions. 
\begin{enumerate}
\item If the weak cone conjecture of $\bMov^e(X_{\bar K}/\bar K)$ holds true, then, after shrinking $S$, there is a finite \'etale Galois morphism $T \to S$ such that for any $\bar g \in \PsAut(X_{\bar K}/\bar K, \De_{\bar K})$, there exists a $g\in \PsAut(\widetilde{X_{T}}/T,\widetilde{\De_T})$ such that $g_{\bar K}$ and $\bar g$ induce the same action on $N^1(\widetilde{X_{\bar K}}/\bar K)_\Rr$. Here $\widetilde{X_{T}}$ is  a small $\Qq$-factorization of $X_T$ and $\widetilde{\De_T}$ is the strict transform of $\De_T$.

\item If the weak cone conjecture of $\bAmp^e(X_{\bar K}/\bar K)$ holds true, then, after shrinking $S$, there is a finite \'etale Galois morphism $T \to S$ such that for any $\bar g \in \Aut(X_{\bar K}/\bar K, \De_{\bar K})$, there exists a $g\in \Aut(X_{T}/T, \De_{T})$ such that $g_{\bar K}$ and $\bar g$ induce the same action on $N^1(X_{\bar K}/\bar K)_\Rr$.
\end{enumerate}
\end{proposition}
\begin{proof}
We only show (1) as (2) can be shown analogously.

By assumption, Conjecture \ref{conj: KM conj} (1) is satisfied for $\Mov(\widetilde{X_{\bar K}}/\bar K)$. As $\widetilde{X_{\bar K}}$ is projective over $\bar K$, if $\pm [D] \in \bEff(\widetilde{X_{\bar K}}/\bar K)$, then we have $D \equiv 0$. This means that  $\bEff(\widetilde{X_{\bar K}}/\bar K)$ is non-degenerate. In particular, $\bMov(\widetilde{X_{\bar K}}/\bar K)$ is also non-degenerate. Lemma \ref{le: existence of fun domain} shows that $\Mov(\widetilde{X_{\bar K}}/\bar K)_+$ admits a rational polyhedral fundamental domain under the action of $\bar \Gamma_B$. By Theorem \ref{thm: finite presented}, $\bar\Gamma_B$ is finitely presented (the following argument only needs it to be finitely generated). Choose 
\[
\bar g_1, \ldots, \bar g_m \in \PsAut(\widetilde{X_{\bar K}}/\bar K, \De_{\bar K})
\] such that $\iota_{\bar K}(\bar g_1), \ldots, \iota_{\bar K}(\bar g_m)$ (see \eqref{eq: tau}) are generators of $\bar\Gamma_B$. As $N^1(\widetilde{X_{\bar K}}/\bar K)_\Rr$ is of finite dimension, let $\widetilde{\bar D_1}, \ldots, \widetilde{\bar D_\rho}$ be divisors such that $[\widetilde{\bar D_1}], \ldots, [\widetilde{\bar D_\rho}]$ is a basis. By Lemma \ref{lem: spread out and specialization}, after shrinking $S$, there is a finite \'etale Galois base change $T \to S$ such that $\bar g_j$ and $\bar D_i \coloneqq \ti \pi_*\widetilde{\bar D_i}$ can be defined on $X_T \to T$. In other words, there exist a $g_j \in \Bir(X_T/T,\De_T)$ and a $D_i$ on $X_T$, such that $(g_j)_{\bar K}=\bar g_j$ and $(D_i)_{\bar K} = \bar D_i$. Shrinking $T$, $(X_T,\De_T)$ has klt singularities by Proposition \ref{prop: Generic property 0} (2). Let $\mu: \widetilde{X_T} \to X_T$ be a small $\Qq$-factorization. Let $\widetilde{D_i} \coloneqq \mu^{-1}_*D_i$. Shrinking $T$ further, there exists a natural inclusion $N^1(\widetilde{X_T}/T)_\Rr \hookrightarrow N^1(\widetilde{(X_T)}_{\bar K}/\bar K)_\Rr$ by Proposition \ref{prop: Generic property}. Because $(\widetilde{(X_T)}_{\bar K},\widetilde{(\De_T)}_{\bar K})$ has klt singularities by Proposition \ref{prop: Generic property 0} (1) and $\widetilde{(X_T)}_{\bar K} \to (X_T)_{\bar K}=X_{\bar K}$ is a small morphism, a small $\Qq$-factorization $Y_{\bar K} \to \widetilde{(X_T)}_{\bar K}$ is still a small $\Qq$-factorization of $X_{\bar K}$. Thus
\[
N^1(\widetilde{X_T}/T)_\Rr \hookrightarrow N^1(\widetilde{(X_T)}_{\bar K}/\bar K)_\Rr \hookrightarrow  N^1(Y_{\bar K}/\bar K)_\Rr \simeq N^1(\widetilde{X_{\bar K}}/\bar K)_\Rr.
\] By the choice of $T$, this is also a surjective map. Hence 
\[
N^1(\widetilde{(X_T)}_{\bar K}/\bar K)_\Rr \simeq  N^1(Y_{\bar K}/\bar K)_\Rr
\] and thus $\widetilde{(X_T)}_{\bar K}$ is $\Qq$-factorial. As $\widetilde{(X_T)}_{\bar K} \to X_{\bar K}$ is a small $\Qq$-factorization, it suffices to show the claim for $N^1(\widetilde{(X_T)}_{\bar K}/\bar K)_\Rr$.

We claim that after shrinking $T$, we have $g_j \in \PsAut(X_T/T,\De_T)$ for each $j$. If $g_j \in \Bir(X_T/T,\De_T)\backslash \PsAut(X_T/T,\De_T)$, then there are finitely many divisors $B_l, l\in J$ which are contracted by $g_j$ or $g^{-1}_j$. As $(g_j)_{\bar K}$ and $(g^{-1}_j)_{\bar K}$ do not contract $(B_l)_{\bar K}$, $B_l$ is vertical over $T$. Therefore, shrinking $T$, we can assume that $g_j$ and $g^{-1}_j$ do not contract divisors for each $j$. This shows the claim.

Finally, let $\overline{\widetilde{D_i}} \coloneqq \widetilde{(D_i)}_{\bar K}$, and for $g, h \in \{g_j \mid 1 \leq j \leq m\}$, let $\bar g \coloneqq g_{\bar K}, \bar h\coloneqq h_{\bar K}$. Then for each $i$,
\[
\bar g_*  (\bar h_*(\overline{\widetilde{D_i}})) =\overline{(g\circ h)}_* (\overline{\widetilde{D_i}}).
\] This implies that 
\[
\iota_{\bar K}(\bar g)(\iota_{\bar K}(\bar h)\cdot [\overline{\widetilde{D_i}}]) = \iota_{\bar K}(\overline{g \cdot h})\cdot [\overline{\widetilde{D_i}}].
\] As $[\overline{\widetilde{D_i}}], i=1, \ldots, \rho$ is a basis of $N^1(\widetilde{(X_T)}_{\bar K}/\bar K)_\Rr$, we have $$\iota_{\bar K}(\bar g)\iota_{\bar K}(\bar h) = \iota_{\bar K}(\overline{g \cdot h}).$$ Now the desired result follows as $\iota_{\bar K}(\bar g_j), 1 \leq j \leq m$ generate $\bar\Gamma_B$.
\end{proof}

\begin{remark}\label{rmk: higher model}
Let $T' \to S$ be a finite \'etale Galois morphism which factors through $T \to S$. By the proof of Proposition \ref{prop: realize actions}, after shrinking $T'$, the claims in Proposition \ref{prop: realize actions} still hold true for $T' \to S$.
\end{remark}

\begin{theorem}
Let $f: (X, \De) \to S$ be a klt Calabi-Yau fiber space. Assume that good minimal models of effective klt pairs exist in dimension $\dim(X/S)$.
\begin{enumerate}
\item Assume that the weak cone conjecture holds true for $\bMov^e(X_{\bar K}/\bar K)$. Then, after shrinking $S$, there is a finite \'etale Galois morphism $T \to S$ such that the cone conjecture holds true for $\bMov^e(X_{T}/T)$. 

\item Assume that the weak cone conjecture holds true for $\bAmp^e(X_{\bar K}/\bar K)$. Then, after shrinking $S$, there is a finite \'etale Galois morphism $T \to S$ such that the cone conjecture holds true for $\bAmp^e(X_{T}/T)$.
\end{enumerate}
\end{theorem}
\begin{proof}
We only show (1) as (2) can be shown analogously. 

By the proof of Proposition \ref{prop: realize actions}, there exist a finite \'etale Galois morphism $T\to S$ and a small $\Qq$-factorization $\widetilde{X_T} \to X_T$ such that $(\widetilde{X_T}, \widetilde{\De_T}) \to T$ is a klt Calabi-Yau fiber space and $(\widetilde{X_T})_{\bar K} \to X_{\bar K}$ is a small $\Qq$-factorization. Replacing $(X, \De) \to S$ by $(\widetilde{X_T}, \widetilde{\De_T}) \to T$, we can assume that $X_{\bar K}$ is $\Qq$-factorial.

Let $\Pi_{\bar K} \subset \bMov^e(X_{\bar K}/\bar K)$ be a rational polyhedral cone such that 
\begin{equation}\label{eq: fun domain}
\PsAut(X_{\bar K}/\bar K, \De_{\bar K}) \cdot \Pi_{\bar K} = \bMov^e(X_{\bar K}/\bar K).
\end{equation} 

By Lemma \ref{lem: spread out and specialization}, after shrinking $S$, there exist a finite \'etale Galois base change $T \to S$ and finitely many effective divisors $D_j, j\in J$ on $X_{T}$ such that $\Cone([(D_j)_{\bar K}] \mid j \in J)=\Pi_{\bar K}$. We can assume that $T \to S$ satisfies Proposition \ref{prop: realize actions} (1) after replacing $T$ by a higher finite \'etale Galois base change (see Remark \ref{rmk: higher model}). We can further assume that $(X_T, \De_T)$ has klt singularities with $K_{X_T}+\De_T\sim_\Rr 0/T$ by Proposition \ref{prop: Generic property 0} (2). 

Let $\mu: \widetilde{X_T} \to X_T$ be a small $\Qq$-factorization and $\widetilde{D_j} \coloneqq \mu_*^{-1}D_j, j\in J$. Shrinking $T$, by Proposition \ref{prop: Generic property} (2), there is a natural inclusion 
\begin{equation}\label{eq: 6.6}
N^1(\widetilde{X_T}/T)_\Rr \hookrightarrow N^1((\widetilde{X_T})_{\bar K}/\bar K)_\Rr.
\end{equation} Because $(\widetilde{X_T})_{\bar K} \to (X_T)_{\bar K}=X_{\bar K}$ is a small morphism and  $X_{\bar K}$ is $\Qq$-factorial, we have $(\widetilde{X_T})_{\bar K} = X_{\bar K}$. By \eqref{eq: 6.6}, we have the natural inclusion
\begin{equation}\label{eq: inj of Mov}
\Mov(\widetilde{X_T}/T) \hookrightarrow \Mov(X_{\bar K}/\bar K).
\end{equation} 

Let $\Pi \coloneqq \Cone([\widetilde{D_j}] \mid j \in J) \subset \Eff(\widetilde{X_T}/T)$.  We claim that
\[
\PsAut(\widetilde{X_T}/T, \De_T) \cdot \Pi \supset \Mov(\widetilde{X_T}/T).
\] In fact, let $[D] \in \Mov(\widetilde{X_T}/T)$. Then there exist a $\bar g \in \PsAut(X_{\bar K}/\bar K,\De_{\bar K})$ and a $[\bar B] \in \Pi_{\bar K}$ such that $\bar g \cdot [\bar B] =[D_{\bar K}]$. By the construction of $\widetilde{X_T}$ and Proposition \ref{prop: realize actions} (1), there exist a $g\in \PsAut(\widetilde{X_T}/T,\widetilde{\De_T})$ and a $\Theta \in \Pi$ such that $[\Theta_{\bar K}]=[\bar B]$ and 
\[
[(g_* \Theta)_{\bar K}]=g_{\bar K}\cdot [\Theta_{\bar K}]=\bar g \cdot [\bar B] =[D_{\bar K}]. 
\] By \eqref{eq: inj of Mov}, $g\cdot [\Theta]=[g_* \Theta]=[D]\in\Mov(\widetilde{X_T}/T)$.

Therefore, Conjecture \ref{conj: shokurov polytope} (1) is satisfied. As $X_{\bar K}$ is projective over $\bar K$, $\pm [D] \in \bEff(X_{\bar K}/\bar K)$ iff $D \equiv 0$. In particular, $\bMov(X_{\bar K}/\bar K)$ is non-degenerate. Then $\bMov(X_T/T)$ is non-degenerate by \eqref{eq: inj of Mov}. Hence, (1) follows from Theorem \ref{thm: main 1} (2).
\end{proof}

It is desirable to deduce the cone conjecture of the Calabi-Yau fiber space $(X, \De) \to S$ from $(X_T, \De_T) \to T$, where $T \to S$ is a finite \'etale  Galois morphism. This seems to be a difficult problem. The main obstacle is to descend elements from $\PsAut(X_T/T,\De_T)$ and $\Aut(X_T/T,\De_T)$ to $\PsAut(X/S,\De)$ and $\Aut(X/S,\De)$. We propose the following question.

\begin{question}\label{que: finite index}
Let $f: X \to S$ be a terminal Calabi-Yau fiber space. Let $T \to S$ be a finite \'etale  Galois morphism. Possibly shrinking $T$, there is a natural group homomorphism $$\PsAut(X/S) \hookrightarrow \PsAut(X_T/T).$$ Let $\Gamma_S$ and $\Gamma_T$ be the images of $\PsAut(X/S)$ and $\PsAut(X_T/T)$ under the group homomorphism $\PsAut(\widetilde{X_T}/T) \to {\rm GL}(N^1(\widetilde{X_T}/T)_\Rr)$. Is $\Gamma_S$ a finite index subgroup of $\Gamma_T$?
\end{question}

A positive answer to Question \ref{que: finite index} would give that the weak cone conjecture for $\bMov^e(X_{\bar K}/\bar K)$ implies that for $\bMov^e(X/S)$.


\end{document}